\documentclass[12pt,reqno]{article}

\usepackage[usenames]{color}
\usepackage{amssymb}
\usepackage{graphicx}
\usepackage{amscd}
\usepackage[colorlinks=true, linkcolor=webgreen, filecolor=webbrown, citecolor=webgreen]{hyperref}
\usepackage{color}
\usepackage{fullpage}
\usepackage{float}
\usepackage{graphics,amsmath,amssymb}
\usepackage{amsthm}
\usepackage{amsfonts}
\usepackage{latexsym}
\usepackage{epsf}
\usepackage{verbatim}
\usepackage{enumerate}
\usepackage{tikz}
\usepackage{caption}
\usepackage{capt-of}
\usepackage{forest}
\usepackage{skak}
\usepackage{mathtools}
\usepackage{array}
\usepackage[blocks]{authblk}
\usepackage{chessfss}
\usepackage{adjustbox}
\usepackage{xcolor}
\usepackage{mathabx}

\usetikzlibrary{matrix}

\definecolor{webgreen}{rgb}{0,.5,0}
\definecolor{webbrown}{rgb}{.6,0,0}
\definecolor{ourwhite}{gray}{0.72}
\definecolor{ourgrey}{gray}{0.43}
\definecolor{ourblack}{gray}{0}
\definecolor{anothergray}{gray}{0.45}

\theoremstyle{plain}
\newtheorem{theorem}{Theorem}

\theoremstyle{definition}

\theoremstyle{remark}

\newcommand{\rookB}[1][1.58ex]{
\adjustbox{Trim=3.2pt 2.2pt 3.2pt 0pt,width=#1,raise=0ex,margin=0.1ex 0ex 0.1ex 0ex}{\BlackRookOnWhite}
}
\newcommand{\knightB}[1][1.85ex]{
\adjustbox{Trim=2.3pt 2.35pt 2.5pt 0pt,width=#1,raise=-0.03ex,margin=0.14ex 0ex 0.14ex 0ex}{\BlackKnightOnWhite}
}
\newcommand{\bishopB}[1][1.79ex]{
\adjustbox{Trim=2.3pt 2pt 2.3pt 0pt,width=#1,raise=-0.12ex,margin=0.1ex 0ex 0.1ex 0ex}{\BlackBishopOnWhite}
}
\newcommand{\queenB}[1][2.05ex]{
\adjustbox{Trim=1.2pt 2.2pt 1.2pt 0pt,width=#1,raise=-0.08ex,margin=0.1ex 0ex 0.1ex 0ex}{\BlackQueenOnWhite}
}
\newcommand{\kingB}[1][1.95ex]{
\adjustbox{Trim=2pt 2pt 2pt 0pt,width=#1,raise=-0.06ex,margin=0.13ex 0ex 0.13ex 0ex}{\BlackKingOnWhite}
}

\newcommand{\bclaim}{\tikz\draw[fill=gray!200!white] (0,0)ellipse (5pt and 5pt);}

\newcommand{\wclaim}{\tikz\draw[fill=gray!0!white] (0,0)ellipse (5pt and 5pt);}

\makeatletter
\newcommand{\thickhline}{
    \noalign {\ifnum 0=`}\fi \hrule height 1pt
    \futurelet \reserved@a \@xhline
}
\newcolumntype{"}{@{\hskip\tabcolsep\vrule width 1pt\hskip\tabcolsep}}
\makeatother

\title{The Struggles of Chessland}
\author{Irene (Hwiseo) Choi}
\author{Shreyas Ekanathan}
\author{Aidan Gao}
\author{Sylvia Zia Lee}
\author{Rajarshi Mandal}
\author{Vaibhav Rastogi}
\author{Daniel Sheffield}
\author{Michael Yang}
\author{Angela Zhao}
\author{Corey Zhao}
\affil{PRIMES STEP}
\author{Tanya Khovanova}
\affil{MIT}

\date{\today}

\begin{document}

\maketitle

\begin{abstract}
This is a fairy tale taking place in Chessland, located in the Bermuda triangle. The chess pieces survey their land and trap enemy pieces. Behind the story, there is fascinating mathematics on how to optimize surveying and trapping. The tale is written by the students in the PRIMES STEP junior group, who were in grades 6 through 9. The paper has a conclusion, written by the group's mentor, Tanya Khovanova, explaining the students' results in terms of graph theory.
\end{abstract}

\section{Chessland}

The Bermuda triangle is a frightening and treacherous place. Over the years, a plethora of boats, planes, ships, and people have vanished without a trace in this lethal zone of the Atlantic Ocean. Despite these mysteries, inside the Bermuda triangle lies a magnificent kingdom of many islands, known as \textit{Chessland} by its inhabitants: chess pieces. Due to a magical spell placed on the kingdom three hundred years ago, Chessland is isolated from the rest of the human world. 

Each island is a square divided into equally-sized square counties. Thus, each island consists of a square number of counties. The islanders do not use compasses because their square islands are all positioned in the same way: their borders are straight lines running parallel cardinal directions. Surprisingly, all islands are of different sizes, and there is an island in each size. An island with $n^2$ counties is called Island $n$. The royal castle is located on Island $1$, where the King and Queen live together with their escort consisting of rooks, bishops, and knights. 

To describe the counties, citizens within each island use coordinates. The south-west county on Island $n$ is given the coordinates $(1,1)$, the south-east is given coordinates $(n,1)$, the north-west and north-east are given coordinates $(1,n)$, and $(n,n)$ respectively. Each county has a color, black or white, based on the eccentric weather patterns: black counties always have cloudy, dark weather, whereas white counties always have clear, light skies. If the sum of the coordinates of a county is odd, it's a white county, and if it is even, it's a black county. This forms a checkerboard pattern. Some of our island diagrams do not include the counties' colors when they are irrelevant. 

Long ago, there lived a paralyzed witch jealous that the others could amble around freely, so she cursed all the residents to move only in oddly specific ways. The witch thought the King had too much power, so she only allowed him to move to any county horizontally, vertically, or diagonally adjacent to its current county. Knights, being brave, were allowed to move in an L shape, jumping two counties in one cardinal direction and one county in a perpendicular direction. Rooks can move to any county that either shares the $x$-coordinate or the $y$-coordinate with the current coordinates of the county the rook is on. Bishops can move to any county located on a line with a slope of $-1$ or $1$ that goes through the county the bishop is currently on. The Queen can move to the union of the set of counties that a rook or bishop in its current coordinates could move to. Every move by someone takes a full day.

\section{Surveying}\label{sec:surveying}

After hearing rumors about enemy infiltration, the King and Queen decide to check every county of their islands to ensure their borders are secure. Being typical rich people, they are indolent, so they have their escort work for them. 

The chess pieces' method of checking is peculiar. Each chess piece, being in a county, can immediately survey all of the counties they can move to. To survey the island, a chess piece has to survey every county as fast as they can. Every day wasted is another opportunity for their enemies to gain foothold in Chessland. Luckily, the King is rich and has plenty of royal hot air balloons. These balloons can drop the chess piece at whichever island and county they want to start at.

\subsection{The Knight's shoe obsession}

The King and Queen need an escort member to visit Island $7$. The Knight, being brave, volunteers to survey the entire island.

The Knight, bowing deeply to the King, sets off to plan the route for the survey of Island $7$. Now, the Knight painstakingly plots out the island, and explores various paths. In the end, he finds the fastest route. ``Your majesty, after hours of time-consuming research, I have concluded that the fastest route to survey Island $7$ takes $11$ days. He shows his plan, pictured in Figure~\ref{fig:knightshoelace7x7}. The Knight can survey all the counties by jumping back and forth between two columns, starting at (3,2), then going to (5,3), (3,4), and so on. But after reaching (3,6), he needs to visit (5,6) and zig-zag back down to visit the remaining counties in the two columns. However, he cannot move from (3,6) to (5,6) in one day. So on the 6th day, he needs a transition county, shown in black, to get to (5,6) and continue zig-zagging down to (5,2).

\begin{figure}[ht!] 
\begin{center}
\begin{tikzpicture}
\draw[step=0.5cm,color=black, line width=1.5] (0,0) grid (3.5,3.5);
\node at (1.25, 2.75) {5};
\node at (2.25, 2.75) {7};
\node at (1.25, 2.25) {8};
\node at (2.25, 2.25) {4};
\node at (1.25, 1.75) {3};
\fill [black] (1.5,1.5) rectangle (2,2);
\node at (1.75, 1.75) {\textcolor{white}{6}};
\node at (2.25, 1.75) {9};
\node at (1.25, 1.25) {10};
\node at (2.25, 1.25) {2};
\node at (1.25, 0.75) {\knight};
\node at (2.25, 0.75) {11};
\end{tikzpicture}
\end{center}
\caption{The plan for Island $7$}
\label{fig:knightshoelace7x7}
\end{figure}
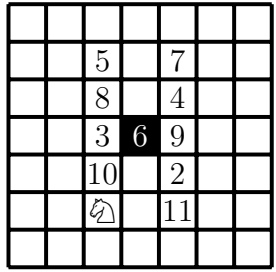

``Well, my Knight, how can you be sure that $11$ days is the best you can do?" the King queries. The Knight presents the diagram in Figure~\ref{fig:knight7x7proof}.
\begin{figure}[ht!] 
\begin{center}
\begin{tikzpicture}
\draw[step=0.5cm,color=black, line width=1.5] (0,0) grid (3.5,3.5);
\node at (0.25,3.25) {$\triangle$};
\node at (1.75,3.25) {\bclaim};
\node at (3.25,3.25) {$\triangle$};
\node at (1.25,2.75) {\wclaim};
\node at (2.25,2.75) {\wclaim};
\node at (2.75,2.75) {\begin{huge}$\blacktriangleright$\end{huge}};
\node at (0.75,2.25) {\wclaim};
\node at (1.75,2.25) {\bclaim};
\node at (2.75,2.25) {\wclaim};
\node at (0.25,1.75) {\bclaim};
\node at (1.25,1.75) {\bclaim};
\node at (2.25,1.75) {\bclaim};
\node at (3.25,1.75) {\bclaim};
\node at (0.75,1.25) {\wclaim};
\node at (1.75,1.25) {\bclaim};
\node at (2.75,1.25) {\wclaim};
\node at (0.75,0.75) {\begin{Huge}$\blacktriangleright$\end{Huge}};
\node at (1.25,0.75) {\wclaim};
\node at (2.25,0.75) {\wclaim};
\node at (0.25,0.25) {$\triangle$};
\node at (1.75,0.25) {\bclaim};
\node at (3.25,0.25) {$\triangle$};
\end{tikzpicture}
\end{center}
\caption{Proof for Island 7}
\label{fig:knight7x7proof}
\end{figure}
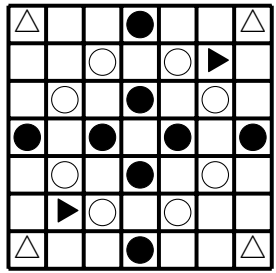

``To survey the counties with white triangles, we will need at least $4$ of the counties marked with white circles to be part of the Knight's route. To survey all of the counties with black triangles, we will need at least $2$ more counties marked with black circles to be part of the Knight's route. All of these counties are white. On a move, the Knight goes from a white county to a black county or vice versa. Therefore, there is at least one intermediate move between each of the $6$ white counties that are part of our route. This gives us a minimum of $6+5=11$.'' the Knight describes. 

``That is brilliant,'' exclaims the King, ``How did you think of that?'' 
 
``I love shoes,'' explains the Knight. 
 
``Huh?'' the King questions. 
 
The Knight responds, ``Connecting the consecutive counties in the plan with lines makes the map look like shoelaces,'' as shown in Figure~\ref{fig:shoelaceexample}.

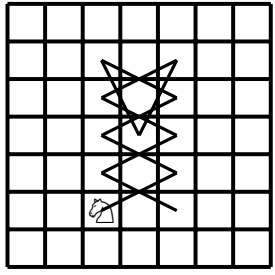
\begin{figure}[ht!] 
\begin{center}
\begin{tikzpicture}[scale=1]
\node at (1.25, 0.75){\knight};
\draw [very thick] (1.25,0.75) -- (2.25,1.25);
\draw [very thick] (2.25,1.25) -- (1.25,1.75);
\draw [very thick] (2.25,2.25) -- (1.25,1.75);
\draw [very thick] (2.25,2.25) -- (1.25,2.75);
\draw [very thick] (1.75,1.75) -- (1.25,2.75);
\draw [very thick] (1.75,1.75) -- (2.25,2.75);
\draw [very thick] (1.25,2.25) -- (2.25,2.75);
\draw [very thick] (1.25,2.25) -- (2.25,1.75);
\draw [very thick] (2.25,1.75) -- (1.25,1.25);
\draw [very thick] (2.25,0.75) -- (1.25,1.25);
\draw[step=0.5cm,color=black, line width=1.5] (0,0) grid (3.5,3.5);
\end{tikzpicture}
\end{center}
\caption{The shoelace pattern for Island $7$}
\label{fig:shoelaceexample}
\end{figure}

The Knight realizes that it is unnecessary to give the pattern a name for a trip to only one island. ``Is the shoelace pattern helpful on other islands?'' asks the Queen, as if she could read the Knight's mind.

The next day, the Knight devises a plan. He reports to the King, ``I can survey all islands whose dimensions are divisible by $7$ with a shoelace pattern.'' He shows an example for Islands $14$ that requires 52 days and is depicted in Figure~\ref{fig:knightsightseeing14}.

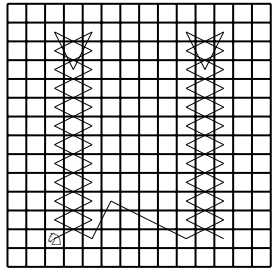
\begin{figure}[ht!]
\begin{center}
\begin{tikzpicture}[scale=0.5]
            \draw[step=0.5cm,color=black, line width=0.75] (0,0) grid (7,7);
            \node at (1.25, 0.75) {\tiny\knight};
            \draw [ultra thin] (1.25,0.75) -- (2.25,1.25);
            \draw [ultra thin] (1.25,1.75) -- (2.25,1.25);
            \draw [ultra thin] (1.25,1.75) -- (2.25,2.25);
            \draw [ultra thin] (1.25,2.75) -- (2.25,2.25);
            \draw [ultra thin] (2.25,3.25) -- (1.25,2.75);
            \draw [ultra thin] (2.25,4.25) -- (1.25,3.75);
            \draw [ultra thin] (2.25,5.25) -- (1.25,4.75);
            \draw [ultra thin] (2.25,6.25) -- (1.25,5.75);
            \draw [ultra thin] (1.25,3.75) -- (2.25,3.25);
            \draw [ultra thin] (1.25,4.75) -- (2.25,4.25);
            \draw [ultra thin] (1.25,5.75) -- (2.25,5.25);
            \draw [ultra thin] (1.75,5.25) -- (2.25,6.25);
            \draw [ultra thin] (1.75,5.25) -- (1.25,6.25);
            \draw [ultra thin] (2.25,5.75) -- (1.25,6.25);
            \draw [ultra thin] (2.25,4.75) -- (1.25,5.25);
            \draw [ultra thin] (2.25,3.75) -- (1.25,4.25);
            \draw [ultra thin] (2.25,2.75) -- (1.25,3.25);
            \draw [ultra thin] (2.25,1.75) -- (1.25,2.25);
            \draw [ultra thin] (2.25,0.75) -- (1.25,1.25);
            \draw [ultra thin] (2.25,5.75) -- (1.25,5.25);
            \draw [ultra thin] (2.25,4.75) -- (1.25,4.25);
            \draw [ultra thin] (2.25,3.75) -- (1.25,3.25);
            \draw [ultra thin] (2.25,2.75) -- (1.25,2.25);
            \draw [ultra thin] (2.25,1.75) -- (1.25,1.25);
            \draw [ultra thin] (2.25,0.75) -- (2.75,1.75);
            \draw [ultra thin] (3.75,1.25) -- (2.75,1.75);
            \draw [ultra thin] (3.75,1.25) -- (4.75,0.75);
            \draw [ultra thin] (4.75,0.75) -- (5.75,1.25);
            \draw [ultra thin] (4.75,1.75) -- (5.75,1.25);
            \draw [ultra thin] (4.75,1.75) -- (5.75,2.25);
            \draw [ultra thin] (4.75,2.75) -- (5.75,2.25);
            \draw [ultra thin] (4.75,1.25) -- (5.75,1.75);
            \draw [ultra thin] (5.75,3.25) -- (4.75,2.75);
            \draw [ultra thin] (5.75,4.25) -- (4.75,3.75);
            \draw [ultra thin] (5.75,5.25) -- (4.75,4.75);
            \draw [ultra thin] (5.75,6.25) -- (4.75,5.75);
            \draw [ultra thin] (4.75,3.75) -- (5.75,3.25);
            \draw [ultra thin] (4.75,4.75) -- (5.75,4.25);
            \draw [ultra thin] (4.75,5.75) -- (5.75,5.25);
            \draw [ultra thin] (5.25,5.25) -- (5.75,6.25);
            \draw [ultra thin] (5.25,5.25) -- (4.75,6.25);
            \draw [ultra thin] (5.75,5.75) -- (4.75,6.25);
            \draw [ultra thin] (5.75,4.75) -- (4.75,5.25);
            \draw [ultra thin] (5.75,3.75) -- (4.75,4.25);
            \draw [ultra thin] (5.75,2.75) -- (4.75,3.25);
            \draw [ultra thin] (5.75,1.75) -- (4.75,2.25);
            \draw [ultra thin] (5.75,0.75) -- (4.75,1.25);
            \draw [ultra thin] (5.75,5.75) -- (4.75,5.25);
            \draw [ultra thin] (5.75,4.75) -- (4.75,4.25);
            \draw [ultra thin] (5.75,3.75) -- (4.75,3.25);
            \draw [ultra thin] (5.75,2.75) -- (4.75,2.25);
\end{tikzpicture}
   \caption{The shoelace pattern for Island $14$}
   \label{fig:knightsightseeing14}
\end{center}
\end{figure}

In one shoelace pattern, the Knight can survey a $7$ by $n$ rectangle of counties at a time. If the rectangle's south-west corner is $(1,1)$, then the Knight will start at $(3,2)$; move to $(5,3)$; $(3,4$); $(5,5)$; and so on, zig-zagging between 2 columns. When the Knight is in row $n-1$, he can move down to $(4, n-3)$, then move back to row $n-1$ and the opposite column. Using this pattern, he can zig-zag all the way down to $(5, 2)$. 

The King asks, ``How do you know that you will come back to $(5,2)$?'' 

The Knight explains, when traveling downwards, he will land in a county on the opposite column across from a county he has moved to on the way up. For example, if he landed on county $(3, x)$ on the way up, he would land on county $(5, x)$ on the way down. Therefore, he will end in the county that is 2 units across from the county he started from, which is (3,2). So, he will always land on (5,2).

The Knight explains that his idea works on islands of all side lengths divisible by 7. He can start at county (3,2), make a shoelace pattern of appropriate height, and finish at county (5,2). Then he can move to counties (6,4) and (8,3). On the next move, he can start a new shoelace in the next 7 by $n$ rectangle at county (10,2).

The King asks, ``How many days does it take for Island $7k$?''

The Knight clarifies that every 7 by $h$ rectangle has a shoelace pattern. He even invents his own theorem.

\begin{theorem}[The shoelace theorem]
One shoelace pattern in a 7 by $h$ column takes $2h-3$ days.
\end{theorem}

\begin{proof}
In each of the two columns in the 7 by $h$ rectangle used by the Knight, he visits $h-2$ counties, which require $2(h-2)=2h-4$ days. One extra county is required to move between the columns, so it takes $2h-3$ days per 7 by $h$ section. 
\end{proof}

Going back to Island $7k$, by the shoelace theorem, each shoelace pattern takes $14k-3$ days per 7 by $7k$ section. In addition, the Knight needs two extra days to transition between neighboring 7 by $7k$ sections. Therefore, it takes 
\[(14k-3)k+2(k-1) = 14k^2-k-2\]
days to survey Island $7k$.

The King then asks, ``But what about other islands of sizes not divisible by 7? Can we do this on Island 13?'' The Knight shows the plan with 2 shoelaces that takes 47 days as in Figure~\ref{fig:knightsightseeing13}.

\begin{figure}[ht!]
    \begin{center}
        \begin{tikzpicture}[scale=0.8]
          \draw[step=0.5cm,color=black, line width=1.5] (-0.5,0.5) grid (6,7);
         \node at (0.75, 1.25) {\knight};
          \draw [very thick] (0.75,1.25) -- (1.75,1.75);
          \draw [very thick] (1.75,1.75) -- (0.75,2.25);
          \draw [very thick] (0.75,2.25) -- (1.75,2.75);
          \draw [very thick] (1.75,2.75) -- (0.75,3.25);
          \draw [very thick] (0.75,3.25) -- (1.75,3.75);
          \draw [very thick] (1.75,3.75) -- (0.75,4.25);
          \draw [very thick] (0.75,4.25) -- (1.75,4.75);
          \draw [very thick] (1.75,4.75) -- (0.75, 5.25);
          \draw [very thick] (0.75,5.25) -- (1.75,5.75);
          \draw [very thick] (1.75,5.75) -- (0.75,6.25);
          \draw [very thick] (0.75,6.25) -- (1.25,5.25);
          \draw [very thick] (1.25,5.25) -- (1.75,6.25); 
          \draw [very thick] (1.75,6.25) -- (0.75,5.75);
          \draw [very thick] (0.75,5.75) -- (1.75,5.25);
          \draw [very thick] (1.75,5.25) -- (0.75,4.75);
          \draw [very thick] (0.75,4.75) -- (1.75,4.25);
          \draw [very thick] (1.75,4.25) -- (0.75,3.75);
          \draw [very thick] (0.75,3.75) -- (1.75,3.25);
          \draw [very thick] (1.75,3.25) -- (0.75,2.75);
          \draw [very thick] (0.75,2.75) -- (1.75,2.25);
          \draw [very thick] (1.75,2.25) -- (0.75,1.75);
          \draw [very thick] (0.75,1.75) -- (1.75,1.25);
          \draw [very thick] (1.75,1.25) -- (2.75,1.75);
          \draw [very thick] (2.75,1.75) -- (3.75,1.25);
          \draw [very thick] (3.75,1.25) -- (4.75,1.75);
          \draw [very thick] (4.75,1.75) -- (3.75,2.25);
          \draw [very thick] (3.75,2.25) -- (4.75,2.75);
          \draw [very thick] (4.75,2.75) -- (3.75,3.25);
          \draw [very thick] (3.75,3.25) -- (4.75,3.75);
          \draw [very thick] (4.75,3.75) -- (3.75,4.25);
          \draw [very thick] (3.75,4.25) -- (4.75,4.75);
          \draw [very thick] (4.75,4.75) -- (3.75,5.25);
          \draw [very thick] (3.75,5.25) -- (4.75,5.75);
          \draw [very thick] (4.75,5.75) -- (3.75,6.25);
          \draw [very thick] (3.75,6.25) -- (4.25,5.25);
          \draw [very thick] (4.25,5.25) -- (4.75,6.25);
          \draw [very thick] (4.75,6.25) -- (3.75,5.75);
          \draw [very thick] (3.75,5.75) -- (4.75,5.25);
          \draw [very thick] (4.75,5.25) -- (3.75,4.75);
          \draw [very thick] (3.75,4.75) -- (4.75,4.25);
          \draw [very thick] (4.75,4.25) -- (3.75,3.75);
          \draw [very thick] (3.75,3.75) -- (4.75,3.25);
          \draw [very thick] (4.75,3.25) -- (3.75,2.75);
          \draw [very thick] (3.75,2.75) -- (4.75,2.25);
          \draw [very thick] (4.75,2.25) -- (3.75,1.75);
          \draw [very thick] (3.75,1.75) -- (4.75,1.25);
        \end{tikzpicture}
    \end{center}
    \caption{The shoelace pattern for Island $13$}
    \label{fig:knightsightseeing13}
\end{figure}
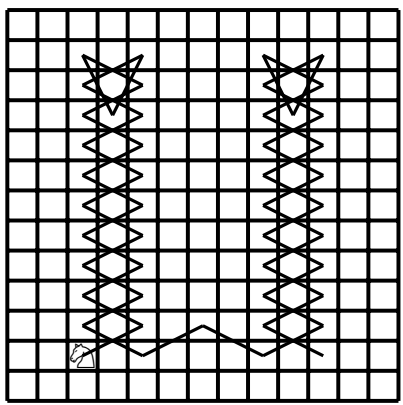

The Knight describes patiently, ``For Island 13, we just move the shoelaces closer to each other. As a bonus, we save one day in transition. For Island $7k-1$, we can use a similar idea, moving the last shoelace one column to the left, to find that we need $2(k-1)- 1 = 2k-3$ days for transitioning between shoelaces, for $k > 1$. Due to the shoelace theorem, we know that one shoelace in a 7 by $7k-1$ rectangle takes $14k-5$ days. Thus, the total number of days is
\[k(14k - 5) + 2k - 3 = 14k^2-3k-3.\]

``Gotcha!'' says the King. ``I am sure this formula doesn't work for $k=1$. By your formula, you have $-1$ transition days in this case.

The Knight is quick to reply, ``Our shoelace is closer to the border, and we can save one day by skipping the last trip, as you can see in Figure~\ref{fig:knightsightseeing6}, where the skipped county is in gray. So, my formula works all the time!''

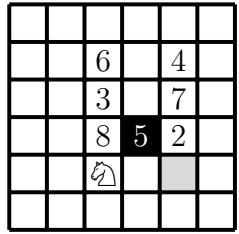
\begin{figure}[ht!]
    \begin{center}   
        \begin{tikzpicture}[scale=1]
            \fill [black] (1.5,1) rectangle (2,1.5);
            \fill [lightgray!60!white] (2,0.5) rectangle (2.5,1);
            \node at (1.75,1.25) {\textcolor{white}{5}};
            \node at (1.25,0.75) {\knight};
			\node at (1.25,1.25) {8};
			\node at (2.25,1.25) {2};
			\node at (1.25,1.75) {3};
			\node at (2.25,1.75) {7};
			\node at (1.25,2.25) {6};
			\node at (2.25,2.25) {4};
            \draw[step=0.5cm,color=black, line width=1.5] (0,0) grid (3,3);
        \end{tikzpicture}
    \end{center}
    \caption{The unfinished shoelace pattern for Island $6$}
    \label{fig:knightsightseeing6}
\end{figure}

``Not so fast!'' says the King. So far, you only know what to do on islands of sizes $7k$ and $7k-1$. What about the rest?

The Knight sums up by explaining that the way he jumps is very complicated; he also has a very bounded vision, so with the shoelace pattern, he covers about 3.5 new counties each day, slightly more. Thus, as a big picture, on islands of size $h$, he needs not more than $\frac{h^2}{3.5}$ days.

Following on, the King asks, ``Still, I want a formula for the number of days that guarantees that you can survey Island $n$.''

The Knight then explains, ``For islands of sizes $7k-r$, where $1 \leq r \leq 6$, we have to fit $k$ shoelaces. This can be done by positioning shoelaces closer to each other. If we remove a column between two adjacent shoelaces, we can reduce the usual two transition counties to just one. By taking $r$ pairs of consecutive shoelaces, removing a column between each pair, and reducing the transition between them from 2 to 1, we can cut the transition time to $2(k-1) - r = 2k-r-2$. Thus, the formula for the total number of days becomes:
\[(2(7k-r) - 3)k + 2k-r-2 = (14k - 2r - 3)k + 2k-r-2  = 14k^2 - (2r+1)k - r - 2.\]

``Are you sure?'' asks the Queen. ``Of course,'' replies the Knight. ``If you plug in $r$ equals 0 or 1, you get the same formula as before.''

``You are being careless again,'' says the Queen. ``Your formula works for large sizes, but what if $k < r$? Then you do not have $r$ pairs of consecutive shoelaces. Suppose we have Island 9, that is $k = 2$, and $r = 5$. Then your formula gives $-3$ transition days! In addition, you have two shoelaces, but you must remove 5 columns.''

``My apologies,'' says the Knight. After some thought, he argues, ``We can always move shoelaces closer to each other or closer to the border. If there are one or three columns between the pair of consecutive shoelaces, I can transition in one day. If there are two or four, I need two days. One transition day is impossible because the counties separated by an even number of columns are of different colors, meaning they need an even number of transition days.''

He continued, ``When $r > k$, we can always remove $r$ columns next to the border or between shoelaces in such a way that we have an odd number of columns between two consecutive shoelaces. Thus, we can always save $\min\{r,k-1\}$ days. Therefore, the formula for the total number of days becomes:
\[(2(7k-r) - 3)k + 2k-\min\{r,k-1\}-2  = 14k^2 - (2r+1)k - \min\{r,k-1\} - 2.\]

``So, how many days do you need for Island 9?'' asks the Queen. 

``I can do it with two shoelaces and one transition day. It will take me $2(2\cdot 9 - 3) + 1 = 31$ days. You can see this plan in Figure~\ref{fig:knightshoelacesislands9-11}. The same figure also shows my plan for Island 11, which takes $2(2\cdot 11 - 3) + 1 = 39$ days.

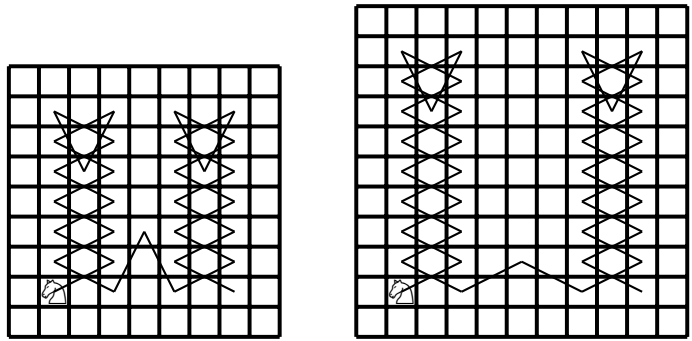
\begin{figure}[ht!]
    \begin{center}
        \begin{tikzpicture}[scale=0.8]
	\node at (0.75, 0.75) {\knight};
            \draw[step=0.5cm,color=black, line width=1.5] (0,0) grid (4.5,4.5);
        \draw [thick] (0.75,0.75) -- (1.75,1.25);
        \draw [thick] (1.75,1.25) -- (0.75,1.75);
        \draw [thick] (0.75,1.75) -- (1.75,2.25);
        \draw [thick] (1.75,2.25) -- (0.75,2.75);
        \draw [thick] (0.75,2.75) -- (1.75,3.25);
        \draw [thick] (1.75,3.25) -- (0.75,3.75);
        \draw [thick] (0.75,3.75) -- (1.25,2.75);
        \draw [thick] (1.25,2.75) -- (1.75,3.75);
        \draw [thick] (1.75,3.75) -- (0.75,3.25);
        \draw [thick] (0.75,3.25) -- (1.75,2.75);
        \draw [thick] (1.75,2.75) -- (0.75,2.25);
        \draw [thick] (0.75,2.25) -- (1.75,1.75);
        \draw [thick] (1.75,1.75) -- (0.75,1.25);
        \draw [thick] (0.75,1.25) -- (1.75,0.75);
        \draw [thick] (1.75,0.75) -- (2.25,1.75);
        \draw [thick] (2.25,1.75) -- (2.75,0.75);
        \draw [thick] (2.75,0.75) -- (3.75,1.25);
        \draw [thick] (3.75,1.25) -- (2.75,1.75);
        \draw [thick] (2.75,1.75) -- (3.75,2.25);
        \draw [thick] (3.75,2.25) -- (2.75,2.75);
        \draw [thick] (2.75,2.75) -- (3.75,3.25);
        \draw [thick] (3.75,3.25) -- (2.75,3.75);
        \draw [thick] (2.75,3.75) -- (3.25,2.75);
        \draw [thick] (3.25,2.75) -- (3.75,3.75);
        \draw [thick] (3.75,3.75) -- (2.75,3.25);
        \draw [thick] (2.75,3.25) -- (3.75,2.75);
        \draw [thick] (3.75,2.75) -- (2.75,2.25);
        \draw [thick] (2.75,2.25) -- (3.75,1.75);
        \draw [thick] (3.75,1.75) -- (2.75,1.25);
        \draw [thick] (2.75,1.25) -- (3.75,0.75);
        \end{tikzpicture}
        \quad\quad
        \begin{tikzpicture}[scale=0.8]
\node at (0.75, 0.75) {\knight};
\draw[step=0.5cm,color=black, line width=1.5] (0,0) grid (5.5,5.5);
\draw [thick] (0.75,0.75) -- (1.75,1.25);
\draw [thick] (1.75,1.25) -- (0.75,1.75);
\draw [thick] (0.75,1.75) -- (1.75,2.25);
\draw [thick] (1.75,2.25) -- (0.75,2.75);
\draw [thick] (0.75,2.75) -- (1.75,3.25);
\draw [thick] (1.75,3.25) -- (0.75,3.75);
\draw [thick] (0.75,3.75) -- (1.75,4.25);
\draw [thick] (1.75,4.25) -- (0.75,4.75);
\draw [thick] (0.75,4.75) -- (1.25,3.75);
\draw [thick] (1.25,3.75) -- (1.75,4.75);
\draw [thick] (1.75,4.75) -- (0.75,4.25);
\draw [thick] (0.75,4.25) -- (1.75,3.75);
\draw [thick] (1.75,3.75) -- (0.75,3.25);
\draw [thick] (0.75,3.25) -- (1.75,2.75);
\draw [thick] (1.75,2.75) -- (0.75,2.25);
\draw [thick] (0.75,2.25) -- (1.75,1.75);
\draw [thick] (1.75,1.75) -- (0.75,1.25);
\draw [thick] (0.75,1.25) -- (1.75,0.75);
\draw [thick] (1.75,0.75) -- (2.75,1.25);
\draw [thick] (2.75,1.25) -- (3.75,0.75);
\draw [thick] (3.75,0.75) -- (4.75,1.25);
\draw [thick] (4.75,1.25) -- (3.75,1.75);
\draw [thick] (3.75,1.75) -- (4.75,2.25);
\draw [thick] (4.75,2.25) -- (3.75,2.75);
\draw [thick] (3.75,2.75) -- (4.75,3.25);
\draw [thick] (4.75,3.25) -- (3.75,3.75);
\draw [thick] (3.75,3.75) -- (4.75,4.25);
\draw [thick] (4.75,4.25) -- (3.75,4.75);
\draw [thick] (3.75,4.75) -- (4.25,3.75);
\draw [thick] (4.25,3.75) -- (4.75,4.75);
\draw [thick] (4.75,4.75) -- (3.75,4.25);
\draw [thick] (3.75,4.25) -- (4.75,3.75);
\draw [thick] (4.75,3.75) -- (3.75,3.25);
\draw [thick] (3.75,3.25) -- (4.75,2.75);
\draw [thick] (4.75,2.75) -- (3.75,2.25);
\draw [thick] (3.75,2.25) -- (4.75,1.75);
\draw [thick] (4.75,1.75) -- (3.75,1.25);
\draw [thick] (3.75,1.25) -- (4.75,0.75);
        \end{tikzpicture}
    \end{center}
    \caption{Shoelaces on Islands 9 and 11}
    \label{fig:knightshoelacesislands9-11}
\end{figure}

``Wait! I found an even better way!'' exclaims the Queen. ``In an example in Figure~\ref{fig:knightsightseeing6} you showed that it is possible to literally cut corners. You can survey these two islands in 24 and 36 days, respectively. My plan is in  Figure~\ref{fig:knightsightseeing9-11}.''

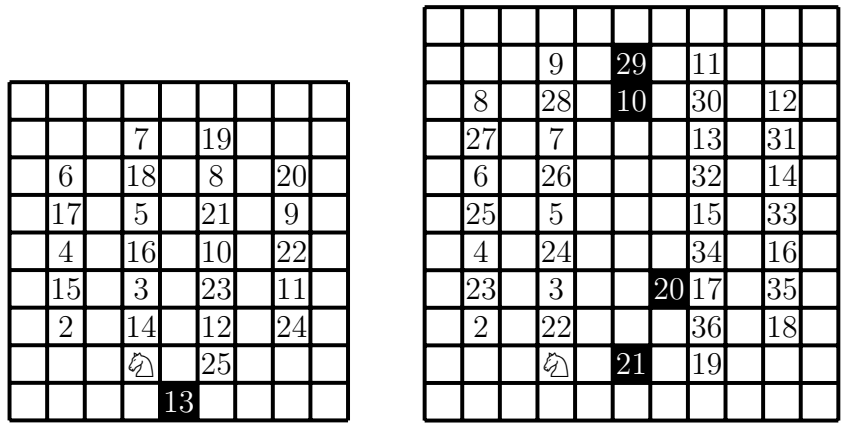
\begin{figure}[ht!]
\begin{center}
\begin{tikzpicture}
	\draw[step=0.5cm,color=black, line width=1.5] (-2.5,-2.5) grid (2,2);
	\node at (-1.75, +0.75){6};
	\node at (-1.75, +0.25){17};
	\node at (-1.75, -0.25){4};
	\node at (-1.75, -0.75){15};
	\node at (-1.75, -1.25){2};
	\node at (-0.75, +1.25){7};
	\node at (-0.75, +0.75){18};
	\node at (-0.75, +0.25){5};
	\node at (-0.75, -0.25){16};
	\node at (-0.75, -0.75){3};
	\node at (-0.75, -1.25){14};
	\node at (-0.75, -1.75){\knight};
	\fill [black] (-0.5,-2.5) rectangle (0,-2);
	\node at (-0.25, -2.25){\textcolor{white}{13}};
	\node at (0.25, +1.25){19};
	\node at (0.25, +0.75){8};
	\node at (0.25, +0.25){21};
	\node at (0.25, -0.25){10};
	\node at (0.25, -0.75){23};
	\node at (0.25, -1.25){12};
	\node at (0.25, -1.75){25};
	\node at (1.25, +0.75){20};
	\node at (1.25, +0.25){9};
	\node at (1.25, -0.25){22};
	\node at (1.25, -0.75){11};
	\node at (1.25, -1.25){24};
\end{tikzpicture}
\quad\quad
\begin{tikzpicture}[scale=1]
            \fill [black] (2.5,0.5) rectangle (3,1);
            \fill [black] (2.5,4) rectangle (3,4.5);
            \fill [black] (2.5,4.5) rectangle (3,5);
            \fill [black] (3,1.5) rectangle (3.5,2);\node at (0.75, 1.25) {2};
            \node at (0.75, 1.75) {23};
            \node at (0.75, 2.25) {4};
            \node at (0.75, 2.75) {25};
            \node at (0.75, 3.25) {6};
            \node at (0.75, 3.75) {27};
            \node at (0.75, 4.25) {8};
            \node at (1.75, 0.75) {\knight};
            \node at (1.75, 1.25) {22};
            \node at (1.75, 1.75) {3};
            \node at (1.75, 2.25) {24};
            \node at (1.75, 2.75) {5};
            \node at (1.75, 3.25) {26};
            \node at (1.75, 3.75) {7};
            \node at (1.75, 4.25) {28};
            \node at (1.75, 4.75) {9};
            \node at (2.75, 0.75) {\textcolor{white}{21}};
            \node at (2.75, 4.25) {\textcolor{white}{10}};
            \node at (2.75, 4.75) {\textcolor{white}{29}};
            \node at (3.25, 1.75) {\textcolor{white}{20}};
            \node at (3.75, 0.75) {19};
            \node at (3.75, 1.25) {36};
            \node at (3.75, 1.75) {17};
            \node at (3.75, 2.25) {34};
            \node at (3.75, 2.75) {15};
            \node at (3.75, 3.25) {32};
            \node at (3.75, 3.75) {13};
            \node at (3.75, 4.25) {30};
            \node at (3.75, 4.75) {11};
            \node at (4.75, 1.25) {18};
            \node at (4.75, 1.75) {35};
            \node at (4.75, 2.25) {16};
            \node at (4.75, 2.75) {33};
            \node at (4.75, 3.25) {14};
            \node at (4.75, 3.75) {31};
            \node at (4.75, 4.25) {12};
            \draw[step=0.5cm,color=black, line width=1.5] (0,0) grid (5.5,5.5);     
\end{tikzpicture}
\end{center}
   \caption{Better plans for Islands 9 and 11}
   \label{fig:knightsightseeing9-11}
\end{figure}

``Why does the Knight like very large islands so much?'' asks the confused King.

The Queen responds, ``Maybe he doesn't want to allude to the fact that he can never fully survey Islands $2$ and $3$ due to his style of moving.'' 

The Knight, overhearing this, explains, ``I showed to you how I can survey any island of size 6 and larger by my lovely plans reminding me of shoes. I can still survey Island 4 and Island 5 in $7$ and $8$ days, respectively, as you can see in Figure~\ref{fig:knight4-5}. I never discussed this before because my plans do not remind me of shoes!''

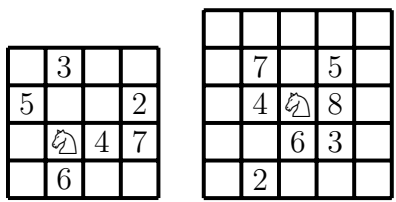
\begin{figure}[ht!]
\begin{center}
\begin{tikzpicture}
\draw[step=0.5cm,color=black, line width=1.5] (0,0) grid (2,2);
\node at (0.75,1.75) {3};
\node at (0.25,1.25) {5};
\node at (1.75,1.25) {2};
\node at (0.75,0.75) {\knight};
\node at (1.25,0.75) {4};
\node at (1.75,0.75) {7};
\node at (0.75,0.25) {6};
\end{tikzpicture}
\quad
\begin{tikzpicture}
\draw[step=0.5cm,color=black, line width=1.5] (-1,-1) grid (1.5,1.5);
\node at (-0.25,+0.75) {7};
\node at (+0.75,+0.75) {5};
\node at (-0.25,+0.25) {4};
\node at (+0.25,+0.25) {\knight};
\node at (+0.75,+0.25) {8};
\node at (+0.25, -0.25) {6};
\node at (+0.75, -0.25) {3};
\node at (-0.25,-0.75) {2};
\end{tikzpicture}
\end{center}
\caption{Knight surveying Island 4 and 5}
\label{fig:knight4-5}
\end{figure}

Table~\ref{tab:knightshoelace} summarizes the Knight's calculations for Islands 1 to 15. The first row is the Island size, followed by the shoelace formula result in the second row. The third row shows the fastest times found by the Knight while experimenting to find a better way than the formula. Note that for Islands 4 and 5, there is not enough room to place a proper shoelace pattern on the island, so the best plan is less optimal than the shoelace formula.

\begin{table}
\begin{center}
    \begin{tabular}{|r|c|c|c|c|c|c|c|c|c|c|c|c|c|c|c|}
\hline
    \textbf{IS}& 1  & 2 &3  &4  &5  &6&7 &8 &9 &10&11&12&13&14&15 \\
    \textbf{SF}&$-1$&1  &3  &5  &7  &9&11&27&31&35&39&43&47&52&83\\
    \textbf{FT}& 1  &N/A&N/A&7  &8  &8&11&20&25&33&36&43&47&52&70\\
\hline
     \end{tabular}
\end{center}
\caption{Numbers from shoelace formula and the fastest times found}
\label{tab:knightshoelace}
\end{table}

The Queen gets excited, ``I noticed that for islands of sizes $7k+1$ you found a much better way than the shoelace formula.'' ``Indeed,'' replies the Knight, ``Instead of squeezing $k+1$ shoelaces, I use $k$ shoelaces with an extra secret pattern. For example, you can see my plan for Island 15 in Figure~\ref{fig:knightsightseeing15}.''

\begin{figure}[ht!]
    \begin{center}
\begin{tikzpicture}
            \fill [black] (1,6) rectangle (1.5,6.5);
            \fill [black] (2,2) rectangle (2.5,2.5);
            \fill [black] (3,1.5) rectangle (3.5,2);
            \fill [black] (4.5,6) rectangle (5,6.5);
            \node at (0.75, 1.25) {\knight};
            \node at (1.75, 1.75) {2};
            \node at (0.75, 2.25) {3};
            \node at (1.75, 2.75) {4};
            \node at (0.75, 3.25) {5};
            \node at (1.75, 3.75) {6};
            \node at (0.75, 4.25) {7};
            \node at (1.75, 4.75) {8};
            \node at (0.75, 5.25) {9};
            \node at (1.75, 5.75) {10};
            \node at (0.75, 6.25) {11};
            \node at (1.75, 6.75) {12};
            \node at (0.75, 7.25) {13};
            \node at (1.25,6.25) {\textcolor{white}{14}};
            \node at (1.75, 7.25) {15};
            \node at (0.75, 6.75) {16};
            \node at (1.75, 6.25) {17};
            \node at (0.75, 5.75) {18};
            \node at (1.75, 5.25) {19};
            \node at (0.75, 4.75) {20};
            \node at (1.75, 4.25) {21};
            \node at (0.75, 3.75) {22};
            \node at (1.75, 3.25) {23};
            \node at (0.75, 2.75) {24};
            \node at (1.75, 2.25) {25};
            \node at (0.75, 1.75) {26};
            \node at (1.75, 1.25) {27};
            \node at (2.25,2.25) {\textcolor{white}{28}};
            \node at (3.25,1.75) {\textcolor{white}{29}};
            \node at (4.25, 1.25) {30};
            \node at (5.25, 1.75) {31};
            \node at (4.25, 2.25) {32};
            \node at (5.25, 2.75) {61};
            \node at (4.25, 3.25) {34};
            \node at (5.25, 3.75) {35};
            \node at (4.25, 4.25) {36};
            \node at (5.25, 4.75) {37};
            \node at (4.25, 5.25) {38};
            \node at (5.25, 5.75) {65};
            \node at (4.25, 6.25) {40};
            \node at (5.25, 6.75) {41};
            \node at (4.25, 7.25) {41};
            \node at (4.75,6.25) {\textcolor{white}{43}};
            \node at (5.25, 7.25) {44};
            \node at (4.25, 6.75) {45};
            \node at (5.25, 6.25) {68};
            \node at (4.25, 5.75) {47};
            \node at (5.25, 5.25) {48};
            \node at (4.25, 4.75) {49};
            \node at (5.25, 4.25) {50};
            \node at (4.25, 3.75) {51};
            \node at (5.25, 3.25) {52};
            \node at (4.25, 2.75) {53};
            \node at (5.25, 2.25) {54};
            \node at (4.25, 1.75) {55};
            \node at (5.25, 1.25) {56};
            \node at (5.75, 2.25) {57};
            \node at (6.75, 1.75) {58};
            \node at (5.75, 1.25) {59};
            \node at (6.25, 2.25) {60};
            \node at (6.25, 3.25) {62};
            \node at (5.75, 4.25) {63};
            \node at (6.25, 5.25) {64};
            \node at (6.25, 6.25) {66};
            \node at (5.75, 7.25) {67};
            \node at (6.25, 6.75) {69};
            \node at (5.75, 7.75) {70};
            \draw[step=0.5cm,color=black, line width=1.5] (-0.5,0.5) grid (7,8);
        \end{tikzpicture}
    \end{center}
    \caption{The remainder 1 pattern for Island $15$}
    \label{fig:knightsightseeing15}
\end{figure}

``I see,'' says the Queen. ``The shoelace pattern is great, but there might be still some room for improvement.''

\subsection{The Rook, the Bishop, and the Queen}

``Ha, ha, ha. You all are so slow; I can do it way faster. I can finish surveying Island $n$ in just $n$ days,'' boasts the Rook.

``How will you be able to achieve such a feat?'' queries the King. 

The Rook responds, ``I will simply move along one row or column to survey the entire Island.'' 

``Can you do better?'' asks the Queen.

``No, and I can prove it! With the first move, I survey one row and one column. Every next move increases the number of surveyed rows or surveyed columns by 1. That means I need at least $n$ days to survey the island.''

The Bishop laughs at the Rook for his slow path. ``I can easily do better than this. You see, I can just start on the major diagonal and visit every county, excluding the two corner counties. Using this, I can cover Island $n$ in just $n-2$ days,'' the Bishop boasts. 

``Well, of course, my Bishop,'' the Rook starts, ``yet, as usual, you demonstrate your oversight. You see, since you can only see counties of the same color as your starting color, you would only cover half of the island!''

``Oh. My apologies. Anyhow, this is no matter. I can bring my brother, who is also a bishop. He traverses on the other color, and will be able to see the other half of the counties, so we will cover the whole island!'', the Bishop explains.

``Yes, yes, this is all very well, but it can only happen if the two major diagonals are of different colors, which happens on even-sized islands. What if you are on an odd-sized island?'' the King queries.

``Okay, I have a strategy for this as well! I call it the `hockey stick' strategy, after my favorite sport. There are two diagonals next to the major diagonal running from the bottom left portion of the island to the top right. I will start on one of these diagonals in county (3,2). From there, I will move onto the top diagonal's county (2,3) and check every county on that top diagonal except for the border counties,'' the Bishop explains. ``My plans for both colors for Island 7 are in Figure~\ref{fig:BishSightseeing}.''

\begin{figure}[ht!]
    \begin{center}
        \begin{tikzpicture}[scale=1]
            \fill [lightgray!50!white] (0,0) rectangle (0.5,0.5);
            \fill [lightgray!50!white] (0,1) rectangle (0.5,1.5);
            \fill [lightgray!50!white] (0,2) rectangle (0.5,2.5);
            \fill [lightgray!50!white] (0,3) rectangle (0.5,3.5);
            \fill [lightgray!50!white] (0.5,0.5) rectangle (1,1);
            \fill [lightgray!50!white] (0.5,1.5) rectangle (1,2);
            \fill [lightgray!50!white] (0.5,2.5) rectangle (1,3);
            \fill [lightgray!50!white] (1,0) rectangle (1.5,0.5);
            \fill [lightgray!50!white] (1,1) rectangle (1.5,1.5);
            \fill [lightgray!50!white] (1,2) rectangle (1.5,2.5);
            \fill [lightgray!50!white] (1,3) rectangle (1.5,3.5);
            \fill [lightgray!50!white] (1.5,0.5) rectangle (2,1);
            \fill [lightgray!50!white] (1.5,1.5) rectangle (2,2);
            \fill [lightgray!50!white] (1.5,2.5) rectangle (2,3);
            \fill [lightgray!50!white] (2,0) rectangle (2.5,0.5);
            \fill [lightgray!50!white] (2,1) rectangle (2.5,1.5);
            \fill [lightgray!50!white] (2,2) rectangle (2.5,2.5);
            \fill [lightgray!50!white] (2,3) rectangle (2.5,3.5);
            \fill [lightgray!50!white] (2.5,0.5) rectangle (3,1);
            \fill [lightgray!50!white] (2.5,1.5) rectangle (3,2);
            \fill [lightgray!50!white] (2.5,2.5) rectangle (3,3);
            \fill [lightgray!50!white] (3,0) rectangle (3.5,0.5);
            \fill [lightgray!50!white] (3,1) rectangle (3.5,1.5);
            \fill [lightgray!50!white] (3,2) rectangle (3.5,2.5);
            \fill [lightgray!50!white] (3,3) rectangle (3.5,3.5);
	    \node at (0.75, 0.75) {\bishop};
            \node at (1.25, 1.25) {2};
            \node at (1.75, 1.75) {3};
            \node at (2.25, 2.25) {4};
            \node at (2.75, 2.75) {5};
            \draw[step=0.5cm,color=black, line width=1.5] (0,0) grid (3.5,3.5);
        \end{tikzpicture}
        \quad\quad
        \begin{tikzpicture}[scale=1]
            \fill [lightgray!50!white] (0,0) rectangle (0.5,0.5);
            \fill [lightgray!50!white] (0,1) rectangle (0.5,1.5);
            \fill [lightgray!50!white] (0,2) rectangle (0.5,2.5);
            \fill [lightgray!50!white] (0,3) rectangle (0.5,3.5);
            \fill [lightgray!50!white] (0.5,0.5) rectangle (1,1);
            \fill [lightgray!50!white] (0.5,1.5) rectangle (1,2);
            \fill [lightgray!50!white] (0.5,2.5) rectangle (1,3);
            \fill [lightgray!50!white] (1,0) rectangle (1.5,0.5);
            \fill [lightgray!50!white] (1,1) rectangle (1.5,1.5);
            \fill [lightgray!50!white] (1,2) rectangle (1.5,2.5);
            \fill [lightgray!50!white] (1,3) rectangle (1.5,3.5);
            \fill [lightgray!50!white] (1.5,0.5) rectangle (2,1);
            \fill [lightgray!50!white] (1.5,1.5) rectangle (2,2);
            \fill [lightgray!50!white] (1.5,2.5) rectangle (2,3);
            \fill [lightgray!50!white] (2,0) rectangle (2.5,0.5);
            \fill [lightgray!50!white] (2,1) rectangle (2.5,1.5);
            \fill [lightgray!50!white] (2,2) rectangle (2.5,2.5);
            \fill [lightgray!50!white] (2,3) rectangle (2.5,3.5);
            \fill [lightgray!50!white] (2.5,0.5) rectangle (3,1);
            \fill [lightgray!50!white] (2.5,1.5) rectangle (3,2);
            \fill [lightgray!50!white] (2.5,2.5) rectangle (3,3);
            \fill [lightgray!50!white] (3,0) rectangle (3.5,0.5);
            \fill [lightgray!50!white] (3,1) rectangle (3.5,1.5);
            \fill [lightgray!50!white] (3,2) rectangle (3.5,2.5);
            \fill [lightgray!50!white] (3,3) rectangle (3.5,3.5);
            \node at (0.75, 1.25) {2};
            \node at (1.25, 0.75) {\bishop};
            \node at (1.25, 1.75) {3};
            \node at (1.75, 2.25) {4};
            \node at (2.25, 2.75) {5};
            \draw[step=0.5cm,color=black, line width=1.5] (0,0) grid (3.5,3.5);
        \end{tikzpicture}
    \end{center}
    \caption{Bishop surveying placement}
    \label{fig:BishSightseeing}
\end{figure}
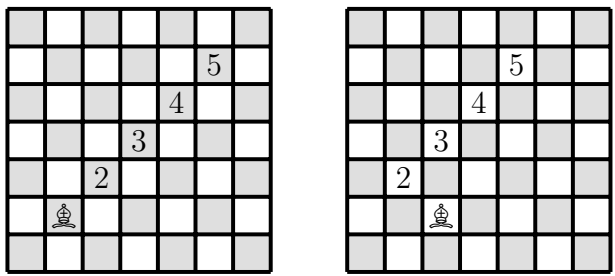

``Can you prove that this number is the minimum?'' the King asks. 

``Yes, I can. Here it is. We need to show that $n-3$ days are not enough for a particular color. Only consider the border rows and columns. For Island $n$, there are $2n-2$ counties of each color along the border. Any bishop can see at most 4 counties on the border, and any move from a particular county can only see at most two new counties on the border.''

``Very elegant! I assume you and your brother will always work in tandem?''

``Yes, of course. We are inseparable,'' the Bishop clarified as he went off to fetch his brother. 

When the Bishop left, the Queen and the Knight were left alone. The Queen was on her way to sleep when the Knight asks, ``Would you like me to help you with surveying, your majesty?''

``Ah, yes! I have already found a wonderful method. I can also move along the main diagonal, similar to the Bishop. However, I can skip more counties than those pathetic bishop brothers. I simply skip certain counties. I call it the `missing beads' method, in reference to a necklace I had that got chewed on by the royal dog. Here's a sketch I made for Island 9.''

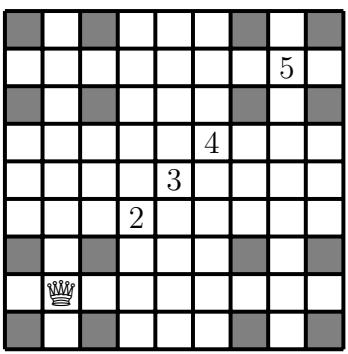
\begin{figure}[h!]
\begin{center}
        \begin{tikzpicture}[scale=1]
            \fill [gray] (0,0) rectangle (0.5,0.5);
            \fill [gray] (0,1) rectangle (0.5,1.5);
            \fill [gray] (0,3) rectangle (0.5,3.5);
            \fill [gray] (0,4) rectangle (0.5,4.5);
            \fill [gray] (1,0) rectangle (1.5,0.5);
            \fill [gray] (1,1) rectangle (1.5,1.5);
            \fill [gray] (1,3) rectangle (1.5,3.5);
            \fill [gray] (1,4) rectangle (1.5,4.5);
            \fill [gray] (3,0) rectangle (3.5,0.5);
            \fill [gray] (3,1) rectangle (3.5,1.5);
            \fill [gray] (3,3) rectangle (3.5,3.5);
            \fill [gray] (3,4) rectangle (3.5,4.5);
            \fill [gray] (4,0) rectangle (4.5,0.5);
            \fill [gray] (4,1) rectangle (4.5,1.5);
            \fill [gray] (4,3) rectangle (4.5,3.5);
            \fill [gray] (4,4) rectangle (4.5,4.5);
	    \node at (0.75, 0.75) {\queen};
            \node at (1.75, 1.75) {2};
            \node at (2.25, 2.25) {3};
            \node at (2.75, 2.75) {4};
            \node at (3.75, 3.75) {5};
            \draw[step=0.5cm,color=black, line width=1.5] (0,0) grid (4.5,4.5);
        \end{tikzpicture}
\caption{Missing beads method on Island 9}
\end{center}
\end{figure}

``Every move I make along the diagonal covers a new row and column. The only counties not belonging to the rows and columns I visited are colored in gray; and I can see those counties diagonally.''

``Fascinating,'' said the Knight, ``Can you continue the pattern for larger islands?''

``Big islands are too much trouble. You can have the Rook survey those.''

``Okay then,'' said the Knight. After a pause, he spoke again, ``Don't you want to find a good consistent strategy?''

``Leave that to the scholars!'' the Queen exclaimed very angrily. ``I've done enough already, you don't have to make me survey massive landmasses!''

Not wanting to face an angry Queen, the Knight slowly backed out of the room.

\subsection{The King}

Finally, the King resolves to survey his islands by himself. In one day, the King can move one county in any direction. He decides on using a spiral pattern, moving around the island toward the center. The spiral consists of straight segments. For example, for Island $n$, he plans to start on (2,2), then move to $(n-1,2)$, then to $(n-1,n-1)$, then to $(2,n-1)$, and then to (2,5). This finishes the first loop of the spiral. He then begins another loop from (2,5) to $(n-4,5)$, and so on. The King plots the route for his venture for Islands 9 and 12, as shown in Figure~\ref{fig:kingstraightspiral}.

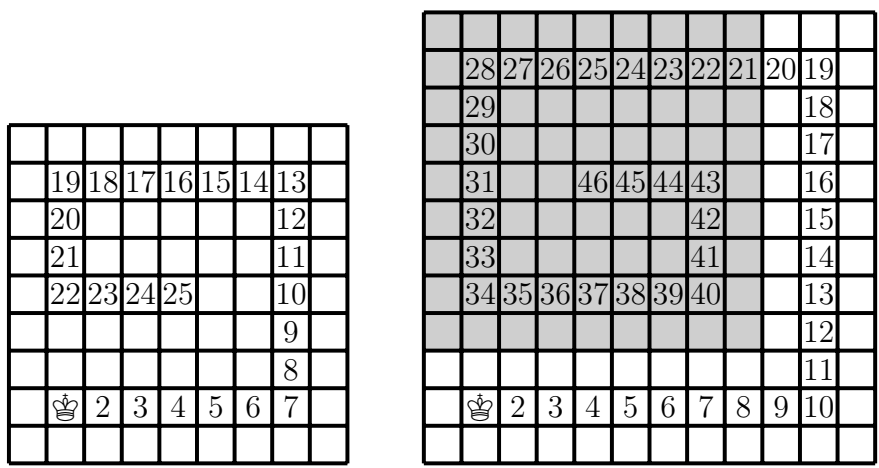
\begin{figure}[ht!]
    \begin{center}
        \begin{tikzpicture}[scale=1]
            \node at (0.75, 0.75) {\king};
            \node at (0.75, 2.25) {22};
            \node at (0.75, 2.75) {21};
            \node at (0.75, 3.25) {20};
            \node at (0.75, 3.75) {19};
            \node at (1.25, 0.75) {2};
            \node at (1.25, 2.25) {23};
            \node at (1.25, 3.75) {18};
            \node at (1.75, 0.75) {3};
            \node at (1.75, 2.25) {24};
            \node at (1.75, 3.75) {17};
            \node at (2.25, 0.75) {4};
            \node at (2.25, 2.25) {25};
            \node at (2.25, 3.75) {16};
            \node at (2.75, 0.75) {5};
            \node at (2.75, 3.75) {15};
            \node at (3.25, 0.75) {6};
            \node at (3.25, 3.75) {14};
            \node at (3.75, 0.75) {7};
            \node at (3.75, 1.25) {8};
            \node at (3.75, 1.75) {9};
            \node at (3.75, 2.25) {10};
            \node at (3.75, 2.75) {11};
            \node at (3.75, 3.25) {12};
            \node at (3.75, 3.75) {13};
            \draw[step=0.5cm,color=black, line width=1.5] (0,0) grid (4.5,4.5);
        \end{tikzpicture}
           \quad\quad \begin{tikzpicture}[scale=1]
           \fill [lightgray!75!white] (0,1.5) rectangle (4.5,6);
            \node at (0.75, 0.75) {\king};
            \node at (0.75, 2.25) {34};
            \node at (0.75, 2.75) {33};
            \node at (0.75, 3.25) {32};
            \node at (0.75, 3.75) {31};
            \node at (0.75, 4.25) {30};
            \node at (0.75, 4.75) {29};
            \node at (0.75, 5.25) {28};
            \node at (1.25, 0.75) {2};
            \node at (1.25, 2.25) {35};
            \node at (1.25, 5.25) {27};
            \node at (1.75, 0.75) {3};
            \node at (1.75, 2.25) {36};
            \node at (1.75, 5.25) {26};
            \node at (2.25, 0.75) {4};
            \node at (2.25, 2.25) {37};
            \node at (2.25, 3.75) {46};
            \node at (2.25, 5.25) {25};
            \node at (2.75, 0.75) {5};
            \node at (2.75, 2.25) {38};
            \node at (2.75, 3.75) {45};
            \node at (2.75, 5.25) {24};
            \node at (3.25, 0.75) {6};
            \node at (3.25, 2.25) {39};
            \node at (3.25, 3.75) {44};
            \node at (3.25, 5.25) {23};
            \node at (3.75, 0.75) {7};
            \node at (3.75, 2.25) {40};
            \node at (3.75, 2.75) {41};
            \node at (3.75, 3.25) {42};
            \node at (3.75, 3.75) {43};
            \node at (3.75, 5.25) {22};
            \node at (4.25, 0.75) {8};
            \node at (4.25, 5.25) {21};
            \node at (4.75, 0.75) {9};
            \node at (4.75, 5.25) {20};
            \node at (5.25, 0.75) {10};
            \node at (5.25, 1.25) {11};
            \node at (5.25, 1.75) {12};
            \node at (5.25, 2.25) {13};
            \node at (5.25, 2.75) {14};
            \node at (5.25, 3.25) {15};
            \node at (5.25, 3.75) {16};
            \node at (5.25, 4.25) {17};
            \node at (5.25, 4.75) {18};
            \node at (5.25, 5.25) {19};
            \draw[step=0.5cm,color=black, line width=1.5] (0,0) grid (6,6);
        \end{tikzpicture}
    \end{center}
    \caption{King's spiral plan for Islands 9 \& 12}
    \label{fig:kingstraightspiral}
\end{figure}

From the diagram, the King knows that it will take him 25 and 46 days to survey Islands 9 and 12, respectively. The King wants to find the number of days this will take on any island, as he doesn't like any unexpected delays.

``What is this gray area?'' asks the Queen.

The King replies, ``If we rotate the path for Island 9, it fits inside the gray area of Island 12. We can do a similar comparison for any two Islands of sizes $n$ and $n+3$. This allows us to create a recursive formula. Suppose $f(n)$ is the number of days needed to survey Island $n$, then
\[f(n+3)= f(n)+2n+3,\]
where $n\ge 3$,'' explains the King.

``Why does it not work when $n$ is equal to 1 or 2?'' asks the Queen.

The King says, ``We enter each smaller Island through county (1,2). The corresponding day is included in the larger Island. If the smaller Island is of size 1 or 2, we do not need any more days; but we still need one day to survey it.'' The King proceeds in his explanation, ``Using the recursive formula, we can start by calculating the number of days for each remainder of 3.'' The King calculates the length of the spiral manually for Islands up to size 8 and puts them in Table~\ref{table:kingsurveyingsmall}. He needs only the values for Islands 3, 4, and 5 to plug in into the recursive formula.

\begin{table}[ht!]
\begin{center}
\begin{tabular}{|r|c|c|c|c|c|c|c|c|}
\hline
\textbf{Island Number}         & 1 & 2 & 3 & 4 & 5 & 6  & 7 &8 \\
\textbf{Days Needed} & 1 & 1 & 1 & 4 & 7 & 10 & 15& 20 \\
\hline
\end{tabular}
\caption{The number of days the King needs to survey for Islands 1--8 using the straight spiral method.}
\label{table:kingsurveyingsmall}
\end{center}
\end{table}

If $n > 2$ is the side length of the island and $n=3k+r$, where $k$ is a non-negative integer, and $r$ is the remainder when $n$ is divided by 3, then 
\[f(n)=3k^2+2rk+r - 2 = \left\lfloor \frac{n^2+2}{3} \right\rfloor - 2.\]
``By the way,'' the King mentions, ``After ignoring the first two values, this sequence can be found in the Online Encyclopedia of Integer Sequences as sequence A071408.''

The Knight interrupts to brag, ``Unfortunately, your majesty, it will take you 25 days on Island 9, while I can survey it in 24 days. Seems that being royalty does not make you good at everything.'' That was a shot which was too painful for the King. Fuming, he leaves for his favorite pub, the \textit{Royal Drink}, hoping to be clear-headed for his journey tomorrow.

The King, however, did not anticipate to be spending 6 hours at the pub playing beer pong. This delayed him in his departure, and he left much later in the day than expected. Additionally, he brought several bottles of his favorite beer brand, as well as his cooler, in order to be able to relax each night with a nice cool beer. 

Waving his wife goodbye, he drunkenly said, ``Worry not, my love! I'll be back in a few weeks!'' With a hearty laugh, he sauntered out of the throne room and into his balloon, setting off for Island 12. Worried, the Queen looked on, wondering how long he would be gone.

Once the King reached the island, he set himself down on county (2,2). He began walking in his spiral path, but to the island's inhabitants, it soon became very clear that his spiral path did not consist of straight segments as he had expected but a very zig-zagged spiral. However, the King continued walking in his ``straight path'', having no idea that the drinks he was consuming each night were preventing him from walking straight. To the King's surprise, he had surveyed all the counties faster than expected! Delighted by this news, he summoned a royal hot air balloon and made his way back to the royal castle. 

The Queen sat, weeks later, pondering upon her husband's decision to drink 5 gallons of beer prior to his departure, when, to her shock, her husband came stumbling back into the courtyard! Delighted, she escorted him back to the throne room, as he seemed to be having trouble walking straight. Must be tired after his long trip, she thought. ``How are you back so early?" the Queen asked wondrously. "I thought you would take longer."

``Well, my dear, it appears that there was an error in my calculations", the King responded. 

``Uh, my lord?'' the Bishop queried. ``It appears that there was no error in your calculations. Your plan would have taken 46 days. However, it seems that you were having a great difficulty walking straight. I heard this news from the inhabitants of Island 12, where you traveled. It seems that you walked in a zig-zag spiral, and this route took 40 days!''

``How can this be? I am very confident that I walked straight. Your nonsensical ideas are not appreciated here. Please leave.''

``But my lord, they also sent me a diagram. Look here,'' the Bishop eagerly said.

\begin{figure}[ht!]
    \begin{center}
        \begin{tikzpicture}[scale=1]
            \node at (0.75, 0.75) {\king};
            \node at (0.75, 2.25) {34};
            \node at (0.75, 3.25) {32};
            \node at (0.75, 4.25) {30};
            \node at (0.75, 5.25) {28};
            \node at (1.25, 1.25) {2};
            \node at (1.25, 2.25) {35};
            \node at (1.25, 2.75) {33};
            \node at (1.25, 3.75) {31};
            \node at (1.25, 4.75) {29};
            \node at (1.25, 5.25) {27};
            \node at (1.75, 0.75) {3};
            \node at (1.75, 2.75) {36};
            \node at (1.75, 4.75) {26};
            \node at (2.25, 1.25) {4};
            \node at (2.25, 3.25) {37};
            \node at (2.25, 5.25) {25};
            \node at (2.75, 0.75) {5};
            \node at (2.75, 2.75) {38};
            \node at (2.75, 4.75) {24};
            \node at (3.25, 1.25) {6};
            \node at (3.25, 2.75) {40};
            \node at (3.25, 3.25) {39};
            \node at (3.25, 5.25) {23};
            \node at (3.75, 0.75) {7};
            \node at (3.75, 4.75) {22};
            \node at (4.25, 1.25) {8};
            \node at (4.25, 5.25) {21};
            \node at (4.75, 0.75) {9};
            \node at (4.75, 1.25) {11};
            \node at (4.75, 2.25) {13};
            \node at (4.75, 3.25) {15};
            \node at (4.75, 4.25) {17};
            \node at (4.75, 4.75) {20};
            \node at (5.25, 0.75) {10};
            \node at (5.25, 1.75) {12};
            \node at (5.25, 2.75) {14};
            \node at (5.25, 3.75) {16};
            \node at (5.25, 4.75) {18};
            \node at (5.25, 5.25) {19};
            \draw[step=0.5cm,color=black, line width=1.5] (0,0) grid (6,6);
        \end{tikzpicture}
    \end{center}
    \caption{King surveying Island 12}
    \label{fig:Kingzigzag12}
\end{figure}

Looking down at the paper handed to him, the King saw that indeed, he had walked in a zig-zag spiral. Astounded by this information, he began pondering as to how and why this route had been faster.

As the King continued thinking, he came across an idea. When he first lands on the island, he can see 9 counties, including the county he's in. When he travels in a straight line, he sees 3 new counties every day. However, the King realizes that by moving diagonally, he sees 5 new counties every day. Unfortunately, he can't survey islands only by diagonal moves. By moving in a zig-zag fashion, starting from the second day, he sees 4 new counties every day. Finally, the King had figured out why his route was faster! Having been drunk, he was unable to walk straight, and he saw more counties per day by moving in a zig-zag manner.

The King decides to implement his new zig-zag spiraling method on other islands too. The King tried to compute how many days the new method would take for any island. However, it looked like a nightmare to calculate! But looking closer, he realizes that it can be simplified. In each zig-zag segment, he can merge the line of counties closer to the island's center with the line of counties further from the center. This forms a straight line segment, so for calculation, he can approximate the zig-zag pattern by a single-line spiral, with a space of 3 counties in between each segment of the spiral.

The length of the new straight spiral, with more distance between the loops, can be calculated similarly to the previous calculation for the straight spiral method: he can compare the spiral for Island $n+4$ with the spiral for Island $n$, as in Figure~\ref{fig:ConcentricSquares}.

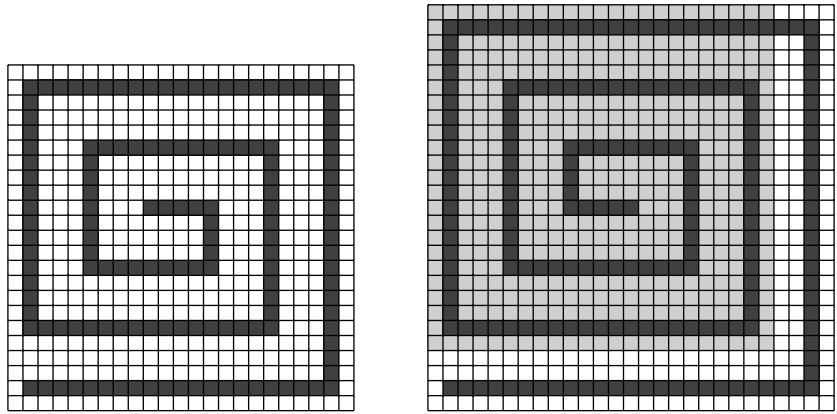
\begin{figure}[htp]
\begin{center}
\begin{tikzpicture}[scale=0.4]

            \fill [darkgray] (0,0) rectangle (10.5,0.5);
            \fill [darkgray] (0.5,10) rectangle (2,10.5);
            \fill [darkgray] (0.5,2) rectangle (2,2.5);
            \fill [darkgray] (0.5,2) rectangle (0,10.5);
            \fill [darkgray] (2,2) rectangle (8.5,2.5);
            \fill [darkgray] (10,10.5) rectangle (10.5,0.5);
            \fill [darkgray] (10.5,10.5) rectangle (2,10);
            \fill [darkgray] (2.5,4) rectangle (2,8);
            \fill [darkgray] (6.5,4) rectangle (2.5,4.5);
            \fill [darkgray] (8.5,2.5) rectangle (8,8.5);
            \fill [darkgray] (2,8) rectangle (8.5,8.5);
            \fill [darkgray] (4,8) rectangle (8.5,8.5);
            \fill [darkgray] (4,6) rectangle (6,6.5);
            \fill [darkgray] (6,4.5) rectangle (6.5,6.5);
            \draw[step=0.5cm,color=black, line width=0.5] (-0.5,-0.5) grid (11,11);
\end{tikzpicture}
\quad
\quad
\begin{tikzpicture}[scale=0.4]
          \fill [lightgray!75!white] (-0.5,1.5) rectangle (11,13);
\fill [darkgray] (0,0) rectangle (12.5,0.5);
            \fill [darkgray] (12,0) rectangle (12.5,12.5);
            \fill [darkgray] (0,12) rectangle (12.5,12.5);
            \fill [darkgray] (0.5,2) rectangle (2,2.5);
            \fill [darkgray] (0.5,2) rectangle (0,12.5);
            \fill [darkgray] (2,2) rectangle (10.5,2.5);
            \fill [darkgray] (10,10.5) rectangle (10.5,2.5);
            \fill [darkgray] (10.5,10.5) rectangle (2,10);
            \fill [darkgray] (2.5,4) rectangle (2,10);
            \fill [darkgray] (2.5,4.5) rectangle (4,4);
            \fill [darkgray] (8.5,4.5) rectangle (4,4);
            \fill [darkgray] (8.5,4.5) rectangle (8,8.5);
            \fill [darkgray] (4,8) rectangle (8.5,8.5);
            \fill [darkgray] (4,8) rectangle (4.5,6);
            \fill [darkgray] (4,8) rectangle (8.5,8.5);
            \fill [darkgray] (4.5,6) rectangle (6,6.5);
            \fill [darkgray] (6,6) rectangle (6.5,6.5);
            \draw[step=0.5cm,color=black, line width=0.5] (-0.5,-0.5) grid (13,13);
\end{tikzpicture}
\end{center}
\caption{Recursion for Island 27}
\label{fig:ConcentricSquares}
\end{figure}

We can create a similar recursive formula to the one before, where $n\ge 3$, and $g(n)$ is the number of days needed for the spiral on Island $n$:
\[g(n+4)= g(n)+2(n+3).\]
The King makes an explicit formula for his new wider spiral. If $n$ is the side length of the island and $n=4k+r$, where $k$ is a non-negative integer, and $r$ is the remainder when $n$ is divided by 4, then
\[g(n)=4k^2+(2r+2)k+r-2.\]
We can rewrite this in terms of $n$ as
\[g(n)=\left\lfloor \frac{n^2+2n+1}{4} \right\rfloor - 2 - t(r),\]
where $t(3) = 1$, and zero for other remainders. If we want to be fancy, we can express $t(r)$ in terms of $n$ as $\left\lfloor \frac{n - 4\lfloor \frac{n}{4}  \rfloor}{3}  \right\rfloor$.

The Queen checks his work. ``I think this simplified spiral is short by a day compared to the actual number of days needed for surveying,'' she says. ``You forgot to include your day in the county at (2, 5). You had to travel here, or else you wouldn't have seen the county at (1, 4).''

``Oops,'' the King says, ``I guess we'll call that an extra county, then. Are there any more extra counties I have to visit?''

``No,'' the Queen says, ``You will not need an extra county in the $n-4$ by $n-4$ square inside Island $n$ because you enter the smaller square at (3, 1), where you can see (4, 1) from. Nor will you need an extra county in the $n-8$ by $n-8$ square, and so on. You only need at most 1 extra county per island.''

``Wait,'' the King says. ``Sometimes the simplified spiral takes the same number of days as the actual number of days needed. I only need an extra day on island sizes with a remainder of 2, 3, 5, 6, and 7 when divided by 8.''

``Why is that?'' the Queen asks.

``Well, in Islands 1, 4, and 8, the spirals contain an extra county. When these spirals are nested into islands with side length 8 counties larger, I can remove the extra day in the smaller spiral. But in order to survey $(1,4)$, I need to visit an extra county in the larger spiral, so it cancels, and I don't need to add an extra day. But Islands 2, 3, 5, 6, and 7 do not contain extra counties. So there is nothing to remove in the smaller spiral, but I need to add a day for the larger spiral. So when these spirals are nested into larger islands, I need to add a day.''

The King creates a new function $G(n)$ that is the actual number of days that he needs for surveying Island $n$ using the zig-zag method. Thus, the formula is
\[G(n) = g(n) + s(r),\]
and 
\[G(n)=\left\lfloor \frac{n^2+2n+1}{4} \right\rfloor - 2 + s(r)-t(r),\]
where $t(3)=s(2)=s(3)=s(5)=s(6)=s(7)=1$, and zero for all other remainders.

Figure~\ref{fig:king11sightseeing} shows an example of how to form the King's actual plan from the spiral.

\begin{figure}[ht!]
    \begin{center}
        \begin{tikzpicture}[scale=1]
            \fill [black] (0.5,2) rectangle (1,2.5);\node at (0.75, 0.75) {\king};
            \node at (0.75, 2.25) {\textcolor{white}{30}};
            \node at (0.75, 2.75) {29};
            \node at (0.75, 3.25) {28};
            \node at (0.75, 3.75) {27};
            \node at (0.75, 4.25) {26};
            \node at (0.75, 4.75) {25};
            \node at (1.25, 0.75) {2};
            \node at (1.25, 2.75) {31};
            \node at (1.25, 4.75) {24};
            \node at (1.75, 0.75) {3};
            \node at (1.75, 2.75) {32};
            \node at (1.75, 4.75) {23};
            \node at (2.25, 0.75) {4};
            \node at (2.25, 2.75) {33};
            \node at (2.25, 4.75) {22};
            \node at (2.75, 0.75) {5};
            \node at (2.75, 2.75) {34};
            \node at (2.75, 4.75) {21};
            \node at (3.25, 0.75) {6};
            \node at (3.25, 4.75) {20};
            \node at (3.75, 0.75) {7};
            \node at (3.75, 4.75) {19};
            \node at (4.25, 0.75) {8};
            \node at (4.25, 4.75) {18};
            \node at (4.75, 0.75) {9};
            \node at (4.75, 1.25) {10};
            \node at (4.75, 1.75) {11};
            \node at (4.75, 2.25) {12};
            \node at (4.75, 2.75) {13};
            \node at (4.75, 3.25) {14};
            \node at (4.75, 3.75) {15};
            \node at (4.75, 4.25) {16};
            \node at (4.75, 4.75) {17};
            \draw[step=0.5cm,color=black, line width=1.5] (0,0) grid (5.5,5.5);
        \end{tikzpicture}
        \quad\quad    
        \begin{tikzpicture}[scale=1]
            \node at (0.75, 0.75) {\king};
            \node at (0.75, 2.25) {30};
            \node at (0.75, 3.25) {28};
            \node at (0.75, 4.25) {26};
            \node at (0.75, 4.75) {25};
            \node at (1.25, 1.25) {2};
            \node at (1.25, 2.25) {31};
            \node at (1.25, 2.75) {29};
            \node at (1.25, 3.75) {27};
            \node at (1.25, 4.25) {24};
            \node at (1.75, 0.75) {3};
            \node at (1.75, 2.25) {32};
            \node at (1.75, 4.75) {23};
            \node at (2.25, 1.25) {4};
            \node at (2.25, 2.75) {33};
            \node at (2.25, 4.25) {22};
            \node at (2.75, 0.75) {5};
            \node at (2.75, 2.25) {34};
            \node at (2.75, 4.75) {21};
            \node at (3.25, 1.25) {6};
            \node at (3.25, 4.25) {20};
            \node at (3.75, 0.75) {7};
            \node at (3.75, 4.75) {19};
            \node at (4.25, 1.25) {8};
            \node at (4.25, 1.75) {11};
            \node at (4.25, 2.75) {13};
            \node at (4.25, 3.75) {15};
            \node at (4.25, 4.25) {18};
            \node at (4.75, 0.75) {9};
            \node at (4.75, 1.25) {10};
            \node at (4.75, 2.25) {12};
            \node at (4.75, 3.25) {14};
            \node at (4.75, 4.25) {16};
            \node at (4.75, 4.75) {17};
            \draw[step=0.5cm,color=black, line width=1.5] (0,0) grid (5.5,5.5);
        \end{tikzpicture}
    \end{center}
    \caption{King surveying Island 11.}
    \label{fig:king11sightseeing}
\end{figure}
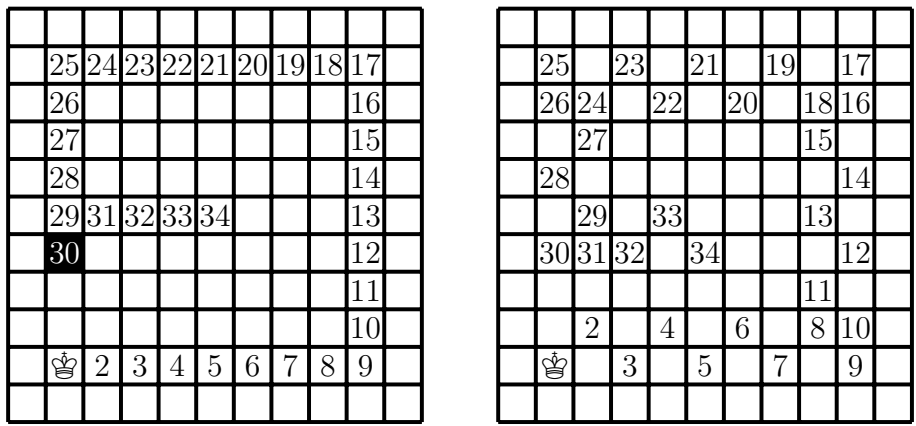

The Queen says, ``There is an issue with smaller islands.''

``For Island 1, the function $G(1)$ gives $-1$, which is not accurate. I actually need 1 day. For Island 2, the formula is correct,'' the King replies. ``And for islands 3--6, we do not need the extra county as the spiral covers the whole island without them: so $g(n)=G(n)$.''

The King manually calculates the number of days he would need for Islands 1 to 15 and puts them in Table~\ref{tab:Island-sightseeing}.

\begin{table}[ht!]
\centering
    \begin{tabular}{|c|c|c|c|c|c|c|c|c|c|c|c|c|c|c|c|}
    \hline
         \textbf{Islands} &1&2&3&4&5&6&7&8&9&10&11&12&13&14&15 \\
         \hline
\textbf{Days} &1&1&1&4&7&10&14&18&23&29&34&40&48&55&62 \\
   \hline
    \end{tabular}
\caption{The number of days the King needs to survey for Islands 1--15 using the zig-zag spiral method.}
\label{tab:Island-sightseeing}
\end{table}

``Do you have a proof that you can't do better?'' the Queen asks. The King replies that he knows his plan is optimal for islands up to size 7. For example, consider Island 6. To survey the corner counties, the King needs to visit each 2 by 2 corner square. He also needs to make at least 3 moves to get from one 2 by 2 corner square to another. Thus, he needs at least 10 days for Island 6, as seen in Figure~\ref{fig:kingsightseeing6}.

\begin{figure}[ht!]
    \begin{center}
        \begin{tikzpicture}[scale=1]
            \node at (0.75, 0.75) {\king};
            \node at (0.75, 2.25) {10};
            \node at (1.25, 2.25) {9};
            \node at (1.75, 1.75) {8};
            \node at (2.25, 2.25) {7};
            \node at (2.25, 1.75) {6};
            \node at (1.75, 1.25) {5};
            \node at (2.25, 0.75) {4};
            \node at (1.75, 0.75) {3};
            \node at (1.25, 1.25) {2};
            
            \draw[step=0.5cm,color=black, line width=1.5] (0,0) grid (3.0,3.0);
        \end{tikzpicture}
        
\end{center}
   \caption{King surveying Island 6}
   \label{fig:kingsightseeing6}
\end{figure}
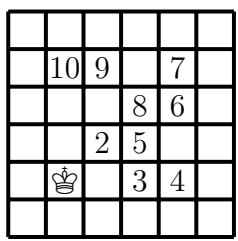

\section{Non-Attacking Trapping}\label{sec:trapping}

Following these surveys, the King identified a potential threat to the well-being of the kingdom: enemy chess pieces, who appeared to be plotting to take over Chessland as territory for their growing kingdom. Upon these observations, the King begins commanding his troops: ``My soldiers, it appears that the enemies are planning something shady, and we do not know what threat they will pose. Each of you will face your enemy counterparts, as you will know how your counterpart moves. Trap the enemy piece so that you see every county that the enemy could move to. Once you have trapped the enemy, there should not remain any counties to which they can move. However, you should not be able to attack the enemy, so it follows that they should not be able to attack you. We would prefer not to enter a full-fledged conflict and merely investigate these nuisances. Additionally, you should not be able to see the county your fellow soldiers are in as it would be unfortunate if you attacked an ally mistaking them for the enemy. We do not have enough soldiers to fend off a full-fledged attack. Because of this, you all have to trap your counterparts using as few soldiers as possible. You have your orders.''

\subsection{Knights}

``My good Knight, my pawns have identified that the enemy chess piece invading Island 7 is in fact an enemy Knight, so I have decided to send you. I need you to be as prepared as possible, so please plan based on any location that the opposing Knight could be. We need to prevent him from moving to any new county so that we can trap him,'' the King tells the Knight.

The Knight nods and goes off to calculate how many soldiers he will need. The Knight knows that he and his Knight soldiers can see at most 8 counties, as shown in Figure~\ref{fig:Knight's Sight}.

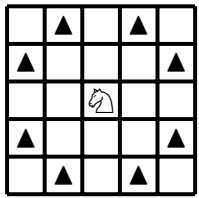
\begin{figure}[ht!]
\begin{center}
\begin{tikzpicture}
\draw[step=0.5cm,color=black, line width=1.5] (0,0) grid (2.5,2.5);
\node at (0.25,0.25) {};
\node at (0.25,0.75) {$\blacktriangle$};
\node at (0.25,1.25) {};
\node at (0.25,1.75) {$\blacktriangle$};
\node at (0.25,2.25) {};
\node at (0.75,0.25) {$\blacktriangle$};
\node at (0.75,0.75) {};
\node at (0.75,1.25) {};
\node at (0.75,1.75) {};
\node at (0.75,2.25) {$\blacktriangle$};
\node at (1.25,0.25) {};
\node at (1.25,0.75) {};
\node at (1.25,1.25) {\knight};
\node at (1.25,1.75) {};
\node at (1.25,2.25) {};
\node at (1.75,0.25) {$\blacktriangle$};
\node at (1.75,0.75) {};
\node at (1.75,1.25) {};
\node at (1.75,1.75) {};
\node at (1.75,2.25) {$\blacktriangle$};
\node at (2.25,0.25) {};
\node at (2.25,0.75) {$\blacktriangle$};
\node at (2.25,1.25) {};
\node at (2.25,1.75) {$\blacktriangle$};
\node at (2.25,2.25) {};
\end{tikzpicture}
\caption{The counties the Knight can see}
\label{fig:Knight's Sight}
\end{center}
\end{figure}

Since a knight can see fewer counties when he is nearer to the boundaries of the island, the Knight divides his task into cases, represented in Figure~\ref{fig:knighttrapping}. For each case, he figures out how many knight soldiers are needed to trap the enemy Knight.

\begin{enumerate}
\item The enemy Knight is at least two counties away from any border.
\item The enemy Knight is one county away from one border and at least two counties away from another border.
\item The enemy Knight is one county away from both borders, and the island size is at least 5.
\item The enemy Knight is one county away from both borders on Island 4.
\item The enemy Knight is on one border and one county away from another border.
\item The enemy Knight is in the corner county.
\end{enumerate}

\begin{figure}[ht!]
\begin{center}
\begin{tikzpicture}
\draw[step=0.5cm,color=black, line width=1.5] (0,0) grid (2.5,2.5);
\node at (0.75,1.75) {\knight};
\node at (1.75,1.75) {\knight};
\node at (1.25,1.25) {\knightB};
\node at (0.75,0.75) {\knight};
\node at (1.75,0.75) {\knight};
\end{tikzpicture}
\begin{tikzpicture}
\draw[step=0.5cm,color=black, line width=1.5] (0,0) grid (2.5,2.5);
\node at (0.75,1.25) {\knight};
\node at (1.75,1.25) {\knight};
\node at (1.25,1.75) {\knightB};
\node at (1.25,0.75) {\knight};
\end{tikzpicture}
\begin{tikzpicture}
\draw[step=0.5cm,color=black, line width=1.5] (0,0) grid (2.5,2.5);
\node at (1.25,0.25) {\knight};
\node at (0.75,1.75) {\knightB};
\node at (2.25,1.25) {\knight};
\end{tikzpicture}
\begin{tikzpicture}
\draw[step=0.5cm,color=black, line width=1.5] (0,0) grid (2,2);
\node at (1.25,1.75) {\knight};
\node at (0.75,1.25) {\knightB};
\node at (0.25,0.75) {\knight};
\node at (1.25,0.75) {\knight};
\end{tikzpicture}
\begin{tikzpicture}
\draw[step=0.5cm,color=black, line width=1.5] (0,0) grid (2,2);
\node at (1.75,1.75) {\knight};
\node at (0.75,1.75) {\knightB};
\node at (1.25,0.25) {\knight};
\end{tikzpicture}
\begin{tikzpicture}
\draw[step=0.5cm,color=black, line width=1.5] (0,0) grid (2,2);
\node at (0.25,1.75) {\knightB};
\node at (1.75,0.25) {\knight};
\end{tikzpicture}
\caption{Trapping the enemy Knight in different cases}
\label{fig:knighttrapping}
\end{center}
\end{figure}
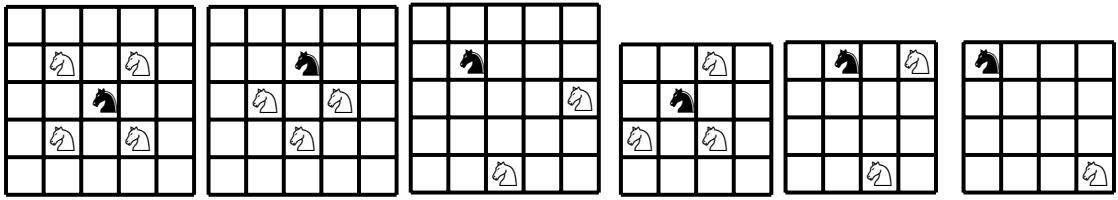

The Knight proudly shows the King how many of his soldiers he will need in order to trap the enemy knight, depending on where it is.

``Hold up, this is getting too confusing! Would you mind organizing this for me?'' 

``Yes, my King,'' the Knight says as he dashes away. After an hour, the Knight returns with his newly organized chart (shown in Figure~\ref{fig:knightanswersheet}) and shows it to the King.

\begin{figure}[ht!]
\begin{center}
\begin{tikzpicture}
\draw[step=0.5cm,color=black, line width=1.5] (0,0) grid (4.5,4.5);
\node at (0.25,0.25) {1};
\node at (0.25,0.75) {2};
\node at (0.25,1.25) {2};
\node at (0.25,1.75) {2};
\node at (0.25,2.25) {:};
\node at (0.25,2.75) {2};
\node at (0.25,3.25) {2};
\node at (0.25,3.75) {2};
\node at (0.25,4.25) {1};
\node at (0.75,0.25) {2};
\node at (0.75,0.75) {2};
\node at (0.75,1.25) {3};
\node at (0.75,1.75) {3};
\node at (0.75,2.25) {:};
\node at (0.75,2.75) {3};
\node at (0.75,3.25) {3};
\node at (0.75,3.75) {2};
\node at (0.75,4.25) {2};
\node at (1.25,0.25) {2};
\node at (1.25,0.75) {3};
\node at (1.25,1.25) {4};
\node at (1.25,1.75) {4};
\node at (1.25,2.25) {:};
\node at (1.25,2.75) {4};
\node at (1.25,3.25) {4};
\node at (1.25,3.75) {3};
\node at (1.25,4.25) {2};
\node at (1.75,0.25) {2};
\node at (1.75,0.75) {3};
\node at (1.75,1.25) {4};
\node at (1.75,1.75) {4};
\node at (1.75,2.25) {:};
\node at (1.75,2.75) {4};
\node at (1.75,3.25) {4};
\node at (1.75,3.75) {3};
\node at (1.75,4.25) {2};
\node at (2.25,0.25) {...};
\node at (2.25,0.75) {...};
\node at (2.25,1.25) {...};
\node at (2.25,1.75) {...};
\node at (2.25,2.25) {$\cdot$};
\node at (2.25,2.75) {...};
\node at (2.25,3.25) {...};
\node at (2.25,3.75) {...};
\node at (2.25,4.25) {...};
\node at (2.75,0.25) {2};
\node at (2.75,0.75) {3};
\node at (2.75,1.25) {4};
\node at (2.75,1.75) {4};
\node at (2.75,2.25) {:};
\node at (2.75,2.75) {4};
\node at (2.75,3.25) {4};
\node at (2.75,3.75) {3};
\node at (2.75,4.25) {2};
\node at (3.25,0.25) {2};
\node at (3.25,0.75) {3};
\node at (3.25,1.25) {4};
\node at (3.25,1.75) {4};
\node at (3.25,2.25) {:};
\node at (3.25,2.75) {4};
\node at (3.25,3.25) {4};
\node at (3.25,3.75) {3};
\node at (3.25,4.25) {2};
\node at (3.75,0.25) {2};
\node at (3.75,0.75) {2};
\node at (3.75,1.25) {3};
\node at (3.75,1.75) {3};
\node at (3.75,2.25) {:};
\node at (3.75,2.75) {3};
\node at (3.75,3.25) {3};
\node at (3.75,3.75) {2};
\node at (3.75,4.25) {2};
\node at (4.25,0.25) {1};
\node at (4.25,0.75) {2};
\node at (4.25,1.25) {2};
\node at (4.25,1.75) {2};
\node at (4.25,2.25) {:};
\node at (4.25,2.75) {2};
\node at (4.25,3.25) {2};
\node at (4.25,3.75) {2};
\node at (4.25,4.25) {1};
\end{tikzpicture}
\end{center}
\caption{Knight answer sheet}
    \label{fig:knightanswersheet}
\end{figure}

``The numbers represent how many soldiers we need based on where the enemy knight is, sir. For example, if the enemy is in county (3,4), we will need 4 knights, but if the enemy is in county (2,2), we will need 2 knights,'' the Knight explains.

The Knight continues, ``We can prove that this is the best we can do using the fact that two knights that aren't attacking each other can simultaneously see at most 2 counties.''

``If the enemy Knight is at least two counties away from the border, he sees 8 counties. Since each of our knights can see at most 2 of these counties, we will need at least 4 knights to restrict the enemy's movement.''

``In general, if the enemy Knight sees $k$ counties, we need at least $\lceil\frac{k}{2}\rceil$ of our knights. As it happens, we can always trap the enemy Knight with $\lceil\frac{k}{2}\rceil$ of our knights.'' 

``This is brilliant! What shall we name this?''

``I call it the \textit{answer sheet},'' the Knight tells the King. The Knight then frowned. 

``What is it?'' asked the King. 

The Knight replied, ``As always, small islands need special attention. On Islands 1 and 2, the enemy is unable to move anywhere, and it is automatically trapped. Also, on Island 3, if the enemy is on the border, we need 2 knights to trap them, and if the enemy is in the center, it is automatically trapped.''

``Well then, I believe we are ready. Thank you, my Knight,'' the King says. The Knight sets off to Island 7 with his fellow knight soldiers on the royal hot air balloon to trap the enemy knight.

\subsection{Bishops}

The King's most trustworthy pawn came in, bursting through the doors. 

``My King, the citizens in many different islands have seen enemy soldiers, and it has been identified that they are invading. The latest report claims that there is an enemy Bishop snooping around Island 7,'' the pawn whimpered.

The King summoned the Bishop and told him the grave news. The Bishop nodded and ran off to start planning right away. The Bishop was carefully making his answer sheet when he realized that there was a huge problem; the Bishop and his soldiers would never be able to trap the enemy Bishop if it was on one of the counties on the main diagonal or the counties that were one county away from the main diagonal, with four exceptions.

``My King, I'm afraid we will not be able to trap the enemy bishop in many counties. The problematic counties are the shaded ones shown in Fig \ref{fig:Problematiccounties},'' the Bishop said sorrowfully.

\begin{figure}[ht!]
\begin{center}
\begin{tikzpicture}
\draw[step=0.5cm,color=black,line width=1.5] (0,0) grid (3.5,3.5);
 \fill [black] (0,0) rectangle (0.5,0.5);
\fill [black] (0.5,0.5) rectangle (1,1);
\fill [black] (1,0.5) rectangle (1.5,1);
\fill [black] (0.5,1) rectangle (1,1.5);
\fill [black] (1,1) rectangle (1.5,1.5);
\fill [black] (1.5,1) rectangle (2,1.5);
\fill [black] (1,1.5) rectangle (1.5,2);
\fill [black] (2,0.5) rectangle (2.5,1);
\fill [black] (2.5,0.5) rectangle (3,1);
\fill [black] (3,0) rectangle (3.5,0.5);
\fill [black] (2,1) rectangle (2.5,1.5);
\fill [black] (2.5,1) rectangle (3,1.5);
\fill [black] (1.5,1.5) rectangle (2,2);
\fill [black] (2,1.5) rectangle (2.5,2);
\fill [black] (0.5,2) rectangle (1,2.5);
\fill [black] (0.5,2.5) rectangle (1,3);
\fill [black] (0,3) rectangle (0.5,3.5);
\fill [black] (1,2) rectangle (1.5,2.5);
\fill [black] (1,2.5) rectangle (1.5,3);
\fill [black] (1.5,2) rectangle (2,2.5);
\fill [black] (2,2) rectangle (2.5,2.5);
\fill [black] (2,2.5) rectangle (2.5,3);
\fill [black] (2.5,2) rectangle (3,2.5);
\fill [black] (2.5,2.5) rectangle (3,3);
\fill [black] (3,3) rectangle (3.5,3.5);
\end{tikzpicture}
\end{center}
\caption{Problematic counties of Island 7}
\label{fig:Problematiccounties}
\end{figure}
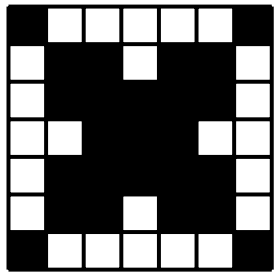

He explained, ``If the enemy is on any of the main diagonals, we cannot see the corners of that diagonal, for the only way we to see them is to be on the main diagonal, which means the enemy can also see us. Additionally, if the enemy is one county away from the main diagonals, but not on the edge of the island, then the 2 bishops that see the ends of the enemy's diagonal adjacent to the main diagonal must see each other. But if the enemy is adjacent to a corner county, then we don't have to see the end of that diagonal since the enemy is on it. So if the enemy is on a main diagonal or one county away from it, with the exception of the counties adjacent to the corners, we cannot trap them.''

``We don't have much time, and you already started your research, so you are our best hope. And, if you can, we don't know if there will be more enemy bishops invading, so please create an answer sheet for other islands as well,'' the King told him. The Bishop hurried back to complete the answer sheets.

After a few hours, the Bishop came back with his plans for Island 7. ``My soldiers and I can trap the enemy Bishop as you can see in Figure~\ref{fig:bishoptrappingplans},'' the Bishop told the King.

\begin{figure}[ht!]
\begin{center}
\begin{tikzpicture}
\draw[step=0.5cm,color=black,line width=1.5] (0,0) grid (3.5,3.5);
\node at (0.75,0.25) {\bishop};
\node at (0.75,3.25) {\bishopB};
\node at (1.75,0.25) {\bishop};
\node at (1.75,3.25) {\bishop};
\node at (2.75,0.25) {\bishop};
\node at (2.75,3.25) {\bishop};
\end{tikzpicture}
\quad
\begin{tikzpicture}
\draw[step=0.5cm,color=black,line width=1.5] (0,0) grid (3.5,3.5);
\node at (0.75,1.75) {\bishop};
\node at (1.25,0.25) {\bishop};
\node at (1.25,3.25) {\bishopB};
\node at (2.75,0.75) {\bishop};
\node at (2.75,2.75) {\bishop};
\end{tikzpicture}
\quad
\begin{tikzpicture}
\draw[step=0.5cm,color=black,line width=1.5] (0,0) grid (3.5,3.5);
\node at (1.75,0.25) {\bishop};
\node at (1.75,1.25) {\bishop};
\node at (1.75,2.25) {\bishop};
\node at (1.75,3.25) {\bishopB};
\end{tikzpicture}
\quad
\begin{tikzpicture}
\draw[step=0.5cm,color=black,line width=1.5] (0,0) grid (3.5,3.5);
\node at (0.75,2.75) {\bishop};
\node at (1.25,0.25) {\bishop};
\node at (1.75,2.75) {\bishopB};
\node at (2.25,0.25) {\bishop};
\node at (2.75,2.75) {\bishop};
\end{tikzpicture}
\end{center}
\caption{Bishop trapping plans}
\label{fig:bishoptrappingplans}
\end{figure}
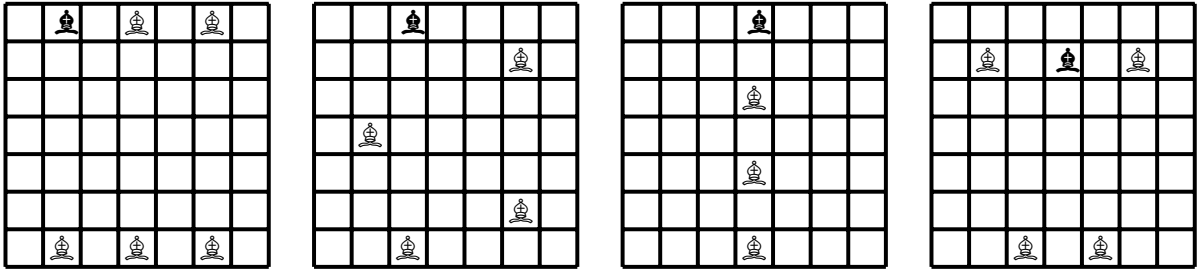

``And for the other cases not listed, I'm assuming you are unable to trap him?''

``Yes, my King. This means that my answer sheet is Figure~\ref{fig:onefourthbishopanswersheet(island7)}.''

\begin{figure}[ht!]
\begin{center}
\begin{tikzpicture}
\matrix (m) [matrix of nodes,
             nodes={draw, minimum size=1em, anchor=center,
                    outer sep=0pt},
             column sep=-\pgflinewidth, row sep=-\pgflinewidth
             ]
{
 0 & 5 & 4 & 3 & 4 & 5 & 0 \\
   & 0 & 0 & 4 & 0 & 0 &   \\
   &   & 0 & 0 & 0 &   &   \\
   &   &   & 0 &   &   &   \\
 };
\end{tikzpicture}
\end{center}
    \caption{One-fourth of Bishop answer sheet for Island 7}
    \label{fig:onefourthbishopanswersheet(island7)}
\end{figure}

``Huh? Why does the answer sheet look like that?'' the King questioned.

``Well, I thought it would be easier for my soldiers to carry. Because of the symmetries, the numbers are repeating anyway. Figure~\ref{fig:onefourthbishopanswersheet(island7)} is the top one-fourth of the entire answer sheet,'' the Bishop explained.

``I see. Did you also make the answer sheet for other islands?''

``Yes, my King. Here is the answer sheet for Island 8 in Figuree~\ref{fig:onefourthbishopanswersheet(island8)}.

\begin{figure}[ht!]
\begin{center}
\begin{tikzpicture}
\matrix (m) [matrix of nodes,
             nodes={draw, minimum size=1em, anchor=center,
                    outer sep=0pt},
             column sep=-\pgflinewidth, row sep=-\pgflinewidth
             ]
{
 0 & 6 & 5 & 4 & 4 & 5 & 6 & 0 \\
   & 0 & 0 & 5 & 5 & 0 & 0 &   \\
   &   & 0 & 0 & 0 & 0 &  &   \\
   &   &   & 0 & 0 &  &   &   \\
 };
\end{tikzpicture}
\end{center}
    \caption{One-fourth of Bishop answer sheet for Island 8}
    \label{fig:onefourthbishopanswersheet(island8)}
\end{figure}

``This looks slightly different,'' the King commented.

``Even and odd-sized islands are always different for bishops. As you can see from these examples, in the middle, there are two counties marked as 0 in the answer sheet for even islands while there is only one for the other,'' the Bishop said.

``Why does that happen?'' 

``It is because on even islands, the middle county does not exist. This means that the two main diagonals of the even islands do not intersect, while in odd islands, they intersect in the middle county. So, in even islands, there will be a 4 by 4 square in the middle where we can never trap the enemy Bishop, and in odd islands, there will be a 3 by 3 square.''

``What about other islands?'' the King asked.

``From these two figures, we can see a pattern that is easy to explain: we need the number of bishops equivalent to the length of the longest diagonal that the enemy can see, minus 1. I can prove that we cannot do better.'' the Bishop says. ``Looking at the longest diagonal that the enemy can see, for every county on that diagonal, we need 1 bishop to see it, as a bishop can see at most 1 county on a diagonal line of counties.''

``I understand that you can't do better, but are you sure you can achieve these numbers?'' the King asked.

The Bishop replied, ``We can position the bishops so that they see all the counties on the main diagonal. We can also position some of them so that they see a county on the shorter diagonal as well. I checked that it is easy to do.''

The King approved the Bishop's plan and sent him and his soldiers to Island 7. There, they found the enemy Bishop lurking in county (7,6). The Bishop called five of his soldiers and commanded them to order themselves as the first diagram in Figure~\ref{fig:bishoptrappingplans}, but rotate 90 degrees clockwise. They were successful in trapping the enemy Bishop, and Island 7 became peaceful again.

\subsection{Rooks}

Amidst all the chaos, the Rook's self-confidence was as high as it had ever been.

``Your plans are so complicated that you need to make whole \textit{answer sheets} for them. How tedious! I need just $n-1$ of my soldiers to trap the enemy rook, independently from where they hide. I can place my soldiers in different rows and columns, as long as they are not the ones that the enemy Rook can see. See my plan for Island 6 in Figure~\ref{fig:rooktrapping}. I can even prove that this is the best I can do!'' he gloats.

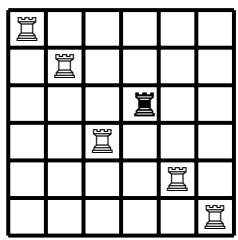
\begin{figure}[ht!]
    \begin{center}
        \begin{tikzpicture}[scale=1]
             \node at (0.25,2.75){\rook};
             \node at (0.75,2.25){\rook};
             \node at (1.25,1.25){\rook};
             \node at (1.75,1.75){\rookB};
             \node at (2.25,0.75){\rook};
             \node at (2.75,0.25){\rook};
            \draw[step=0.5cm,color=black, line width=1.5] (0,0) grid (3,3);
        \end{tikzpicture}
    \end{center}
    \caption{A possible rook formation on Island 6}
    \label{fig:rooktrapping}
\end{figure}

``Go on,'' the King said dully, tired of the Rook's constant boasting.

``Two rooks can survey at most 2 counties together without attacking each other. A rook sees $2n-2$ counties, so to see all the counties that the enemy Rook sees, the number of rooks needed for trapping can't be less than $\frac{2n-2}{2}=n-1$.''

``Why are you sure you can always trap?'' asked the King.

``It is trivial,'' bragged the Rook. ``Imagine that I removed the row and column where the enemy Rook is hiding. Then I can place my rooks on the main diagonal of the leftover square.''

Unfortunately for the Rook, his boasting only made the King assign his troops to trap enemy pieces on many different islands, for his strategy was one of the simpler ones. It would be a while until anyone saw the rooks in the castle again.

\subsection{Kings}

The peace on Island 7 did not last long, as the enemy King started to invade Island 7 again. The white King decides to step in with his brothers. Working on his answer sheet, he initially thought that the sheet would be very similar to the one his Knight made. He felt that he and his Knight could never make more than 8 different moves. However, the King soon realized that his answer sheet was similar to the Bishop's sheet too, as there were counties where the enemy King could not be trapped.

``My brothers and I will never be able to trap the enemy King if he is one county away from a border, nor when he is on the border next to a corner!'' The King shows Figure~\ref{fig:kingonecellawayfromborder} to his escort.

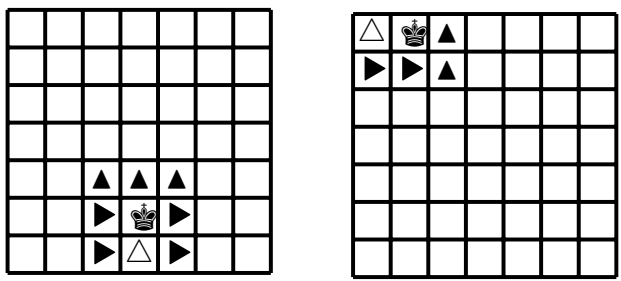
\begin{figure}[ht!]
\begin{center}
\begin{tikzpicture}
\draw[step=0.5cm,color=black,line width=1.5] (0,0) grid (3.5,3.5);
\node at (1.25,0.25) {\begin{Huge}$\blacktriangleright$\end{Huge}};
\node at (1.25,0.75) {\begin{Huge}$\blacktriangleright$\end{Huge}};
\node at (1.25,1.25) {$\blacktriangle$};
\node at (1.75,0.25) {$\triangle$};
\node at (1.75,0.75) {\kingB};
\node at (1.75,1.25) {$\blacktriangle$};
\node at (2.25,0.25) {\begin{Huge}$\blacktriangleright$\end{Huge}};
\node at (2.25,0.75) {\begin{Huge}$\blacktriangleright$\end{Huge}};
\node at (2.25,1.25) {$\blacktriangle$};
\end{tikzpicture}
\quad
\quad
\begin{tikzpicture}
\draw[step=0.5cm,color=black,line width=1.5] (0,0) grid (3.5,3.5);
\node at (0.25,2.75) {\begin{Huge}$\blacktriangleright$\end{Huge}};
\node at (0.25,3.25) {$\triangle$};
\node at (0.75,2.75) {\begin{Huge}$\blacktriangleright$\end{Huge}};
\node at (0.75,3.25) {\kingB};
\node at (1.25,2.75) {$\blacktriangle$};
\node at (1.25,3.25) {$\blacktriangle$};
\end{tikzpicture}
\end{center}
\caption{Enemy king can't be trapped}
\label{fig:kingonecellawayfromborder}
\end{figure}

``The triangles represent where the enemy King can move to, and the black triangles, I can block. It's the white triangle county that I have a problem with. To block the enemy King from moving there, I need to be on one of the tilted black triangles, but it would mean that we are attacking each other,'' the King clarifies. He added, ``But other than these two cases, I believe I will be able to corner the enemy King.''

``Indeed. Now, what other cases remain?'' asked a pawn.

``I suppose we need to consider three more cases; one when the enemy King is on the corner counties, one when he is on the border counties (excluding the corners and the ones adjacent to them), and when he is at least two counties away from every border. You can see my plans in Figure~\ref{fig:kingtrapping}.''

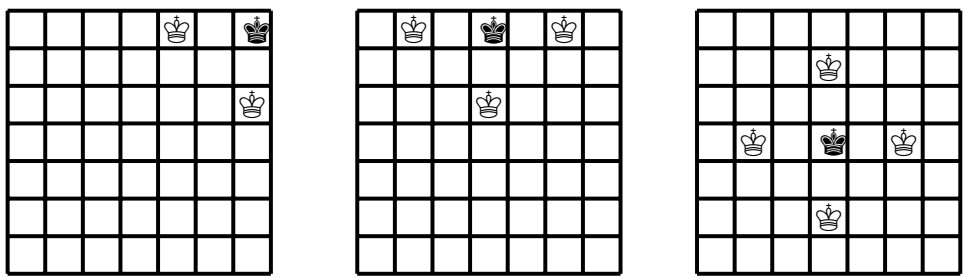
\begin{figure}[ht!]
\centering
\begin{tikzpicture}
\draw[step=0.5cm,color=black,line width=1.5] (0,0) grid (3.5,3.5);
\node at (2.25,3.25) {\king};
\node at (3.25,2.25) {\king};
\node at (3.25,3.25) {\kingB};
\end{tikzpicture}
\quad
\quad
\begin{tikzpicture}
\draw[step=0.5cm,color=black,line width=1.5] (0,0) grid (3.5,3.5);
\node at (0.75,3.25) {\king};
\node at (1.75,2.25) {\king};
\node at (1.75,3.25) {\kingB};
\node at (2.75,3.25) {\king};
\end{tikzpicture}
\quad
\quad
\begin{tikzpicture}
\draw[step=0.5cm,color=black,line width=1.5] (0,0) grid (3.5,3.5);
\node at (0.75,1.75) {\king};
\node at (1.75,0.75) {\king};
\node at (1.75,1.75) {\kingB};
\node at (1.75,2.75) {\king};
\node at (2.75,1.75) {\king};
\end{tikzpicture}
\caption{King trapping cases}
\label{fig:kingtrapping}
\end{figure}

``This is great!'' the King says joyfully. ``My analysis is easy to extrapolate for any island. So, my answer sheet would be Figure~\ref{fig:kinganswersheet}.''

\begin{figure}[ht!]
\begin{center}
\begin{tikzpicture}
\draw[step=0.5cm,color=black,line width=1.5] (0,0) grid (4.5,4.5);
\node at (0.25,0.25) {2};
\node at (0.25,0.75) {0};
\node at (0.25,1.25) {3};
\node at (0.25,1.75) {3};
\node at (0.25,2.25) {:};
\node at (0.25,2.75) {3};
\node at (0.25,3.25) {3};
\node at (0.25,3.75) {0};
\node at (0.25,4.25) {2};
\node at (0.75,0.25) {0};
\node at (0.75,0.75) {0};
\node at (0.75,1.25) {0};
\node at (0.75,1.75) {0};
\node at (0.75,2.25) {:};
\node at (0.75,2.75) {0};
\node at (0.75,3.25) {0};
\node at (0.75,3.75) {0};
\node at (0.75,4.25) {0};
\node at (1.25,0.25) {3};
\node at (1.25,0.75) {0};
\node at (1.25,1.25) {4};
\node at (1.25,1.75) {4};
\node at (1.25,2.25) {:};
\node at (1.25,2.75) {4};
\node at (1.25,3.25) {4};
\node at (1.25,3.75) {0};
\node at (1.25,4.25) {3};
\node at (1.75,0.25) {3};
\node at (1.75,0.75) {0};
\node at (1.75,1.25) {4};
\node at (1.75,1.75) {4};
\node at (1.75,2.25) {:};
\node at (1.75,2.75) {4};
\node at (1.75,3.25) {4};
\node at (1.75,3.75) {0};
\node at (1.75,4.25) {3};
\node at (2.25,0.25) {...};
\node at (2.25,0.75) {...};
\node at (2.25,1.25) {...};
\node at (2.25,1.75) {...};
\node at (2.25,2.25) {...};
\node at (2.25,2.75) {...};
\node at (2.25,3.25) {...};
\node at (2.25,3.75) {...};
\node at (2.25,4.25) {...};
\node at (2.75,0.25) {3};
\node at (2.75,0.75) {0};
\node at (2.75,1.25) {4};
\node at (2.75,1.75) {4};
\node at (2.75,2.25) {:};
\node at (2.75,2.75) {4};
\node at (2.75,3.25) {4};
\node at (2.75,3.75) {0};
\node at (2.75,4.25) {3};
\node at (3.25,0.25) {3};
\node at (3.25,0.75) {0};
\node at (3.25,1.25) {4};
\node at (3.25,1.75) {4};
\node at (3.25,2.25) {:};
\node at (3.25,2.75) {4};
\node at (3.25,3.25) {4};
\node at (3.25,3.75) {0};
\node at (3.25,4.25) {3};
\node at (3.75,0.25) {0};
\node at (3.75,0.75) {0};
\node at (3.75,1.25) {0};
\node at (3.75,1.75) {0};
\node at (3.75,2.25) {:};
\node at (3.75,2.75) {0};
\node at (3.75,3.25) {0};
\node at (3.75,3.75) {0};
\node at (3.75,4.25) {0};
\node at (4.25,0.25) {2};
\node at (4.25,0.75) {0};
\node at (4.25,1.25) {3};
\node at (4.25,1.75) {3};
\node at (4.25,2.25) {:};
\node at (4.25,2.75) {3};
\node at (4.25,3.25) {3};
\node at (4.25,3.75) {0};
\node at (4.25,4.25) {2};
\end{tikzpicture}
\end{center}
\caption{King trapping answer sheet}
\label{fig:kinganswersheet}
\end{figure}

``I just realized that if the enemy King attacks $n$ cells and it is possible to trap it, it will take exactly $\frac{n+1}{2}$ kings to trap it.''

A pawn commented, ``The number of kings needed for trapping is exactly the same as the number of non-diagonal cells the enemy King can attack.''

``You are right, but I am in a hurry to save my country,'' retorted the King. ``I will leave with haste alongside my brothers.'' When they returned, the King received an urgent message from Island 5, stating that an enemy Queen was invading this time.

\subsection{Queens}

``What, no! I can't just go trap an enemy Queen! I have so many other things to do! Yesterday, I spent six hours just choosing what I'm going to wear today! \textit{Six hours!} Do you not see how busy I am?'' the Queen exclaimed in shock.

``You don't have to go; you only have to help us plan. Your sisters will trap the enemy Queen!'' the King clarified.

``I'll think about it,'' the Queen huffed. 

``But the enemy Queen is invading right now!'' 

``Ugh, fine! It's only for this time,'' the Queen said, exasperated. 

``Thank you, dear, you can work on the answer sheet for smaller islands, maybe up to Island 5, and we will work on larger ones.''

``Why do I have to find the answer sheet for all the small ones? The enemy Queen's only invading Island 5.''

``Ah, we thought it would be useful to have those answer sheets because we don't know if they will be invading again. Now, let's get to work!'' the King shouted.

The Queen started to work on Island 2, but she realized her sisters would not be able to even set foot there as the enemy queen attacked every other county. So, she began to work on Island 3. 

``Huh, it seems like my sisters will be unable to trap the enemy queen,'' the Queen muttered to herself. The queen quickly realized that it was impossible to trap the enemy queen on Island 3. The three cases where the enemy queen is in the corner, side, or middle are shown in Figure~\ref{fig:queentrapping3} together with the answer sheet. In the first two cases the enemy queen can always escape to a gray corner. In the third case the enemy queen attacks the whole island.

\begin{figure}[ht!]
\begin{center}
    \begin{tikzpicture}
    \fill [lightgray] (0.5,1) rectangle (1,1.5);
    \draw[step=0.5 cm,color=black, line width=1.5] (0,0) grid (1.5,1.5);
	\node at (1.25,0.75) {\queenB};	
	\node at (0.25,0.25) {\queen};
    \end{tikzpicture}
\quad
\quad
    \begin{tikzpicture}
    \fill [lightgray] (0.5,1) rectangle (0,1.5);
    \draw[step=0.5 cm,color=black, line width=1.5] (0,0) grid (1.5,1.5);
	\node at (1.25,0.75) {\queen};	
	\node at (0.25,0.25) {\queenB};
    \end{tikzpicture}
\quad
\quad
    \begin{tikzpicture}
    \draw[step=0.5 cm,color=black, line width=1.5] (0,0) grid (1.5,1.5);
	\node at (0.75,0.75) {\queenB};	
    \end{tikzpicture}
\quad
\quad
    \begin{tikzpicture}
    \draw[step=0.5 cm,color=black, line width=1.5] (0,0) grid (1.5,1.5);
    \node at (0.25,1.25) {0};
    \node at (0.75,1.25) {0};
    \node at (1.25,1.25) {0};
    
    \node at (0.25,0.75) {0};
    \node at (0.75,0.75) {0};
    \node at (1.25,0.75) {0};
    
    \node at (0.25,0.25) {0};
    \node at (0.75,0.25) {0};
    \node at (1.25,0.25) {0};
    \end{tikzpicture}
\end{center}
    \caption{Queen trapping for Island 3}
    \label{fig:queentrapping3}
\end{figure}
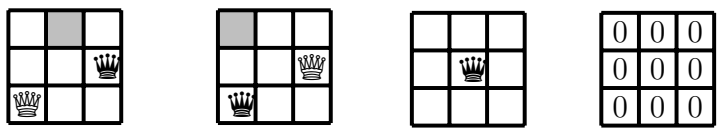

The Queen then begins planning for Island 4, and runs into a hurdle.

``What?! It is impossible to trap the enemy Queen if she's in one of the middle four counties!'' She draws her trapping solutions and the answer sheet for Island 4 in Figure~\ref{fig:queentrapping4}.

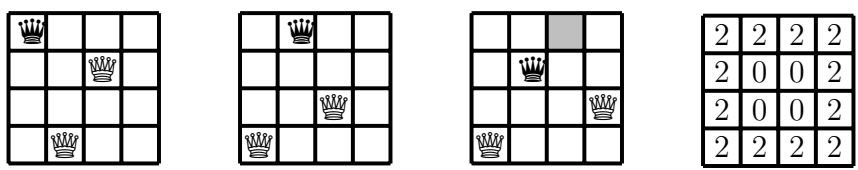
\begin{figure}[ht!]
\begin{center}
\begin{tikzpicture}
\draw[step=0.5cm,color=black, line width=1.5] (0,0) grid (2,2);
\node at (0.25,1.75) {\queenB};
\node at (1.25,1.25) {\queen};
\node at (0.75,0.25) {\queen};
\end{tikzpicture}
\quad
\quad
\begin{tikzpicture}
\draw[step=0.5cm,color=black, line width=1.5] (0,0) grid (2,2);
\node at (0.75,1.75) {\queenB};
\node at (0.25,0.25) {\queen};
\node at (1.25,0.75) {\queen};
\end{tikzpicture}
\quad
\quad
\begin{tikzpicture}
\fill [lightgray] (1.5,2) rectangle (1,1.5);
\draw[step=0.5cm,color=black, line width=1.5] (0,0) grid (2,2);
\node at (0.75,1.25) {\queenB};
\node at (0.25,0.25) {\queen};
\node at (1.75,0.75) {\queen};
\end{tikzpicture}
\quad
\quad
\begin{tikzpicture}
\draw[step=0.5cm,color=black, line width=1.5] (0,0) grid (2,2);
\node at (0.25,1.75) {2};
\node at (0.75,1.75) {2};
\node at (1.25,1.75) {2};
\node at (1.75,1.75) {2};

\node at (0.25,1.25) {2};
\node at (0.75,1.25) {0};
\node at (1.25,1.25) {0};
\node at (1.75,1.25) {2};

\node at (0.25,0.75) {2};
\node at (0.75,0.75) {0};
\node at (1.25,0.75) {0};
\node at (1.75,0.75) {2};

\node at (0.25,0.25) {2};
\node at (0.75,0.25) {2};
\node at (1.25,0.25) {2};
\node at (1.75,0.25) {2};
\end{tikzpicture}
\end{center}
    \caption{Queen trapping for Island 4}
    \label{fig:queentrapping4}
\end{figure}

``I wonder if this will be true for other islands as well,'' the Queen pondered. After some time thinking, she found out that this is only a problem for smaller islands because the number of cells the enemy Queen can attack from one certain county is greater than the ones that are not, making trapping her inflexible.

``Well, I'm tired, so it's time to choose my outfit for tomorrow!'' The Queen hurried to her closet, only to find the King waiting there.

``Aha! I knew you'd come here! Did you finish your work on small islands?''

``Um, I finished Islands 2, 3, and 4,'' the Queen stammered.

``Could you at least finish Island 5 before you choose your clothes? Island 5 is in peril, as you well know.''

``Alright,'' the Queen sighed as she walked back.

She resumed her work, making the answer sheet for Island 5. To her surprise, her sisters would always be able to trap the enemy Queen! Her answer sheet is shown in Figure~\ref{fig:queentrapping5}.

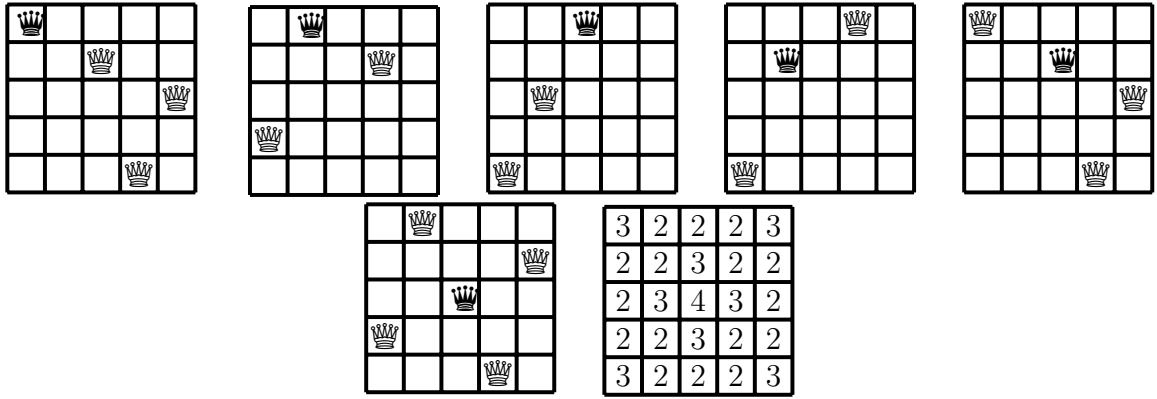
\begin{figure}[ht!]
\begin{center}
\begin{tikzpicture}
\draw[step=0.5cm,color=black, line width=1.5] (0,0) grid (2.5,2.5);
\node at (0.25,2.25) {\queenB};
\node at (1.25,1.75) {\queen};
\node at (2.25,1.25) {\queen};
\node at (1.75,0.25) {\queen};
\end{tikzpicture}
\quad
\begin{tikzpicture}
\draw[step=0.5cm,color=black, line width=1.5] (0,0) grid (2.5,2.5);
\node at (0.75,2.25) {\queenB};
\node at (1.75,1.75) {\queen};
\node at (0.25,0.75) {\queen};
\end{tikzpicture}
\quad
\begin{tikzpicture}
\draw[step=0.5cm,color=black, line width=1.5] (0,0) grid (2.5,2.5);
\node at (1.25,2.25) {\queenB};
\node at (0.75,1.25) {\queen};
\node at (0.25,0.25) {\queen};
\end{tikzpicture}
\quad
\begin{tikzpicture}
\draw[step=0.5cm,color=black, line width=1.5] (0,0) grid (2.5,2.5);
\node at (1.75,2.25) {\queen};
\node at (0.75,1.75) {\queenB};
\node at (0.25,0.25) {\queen};
\end{tikzpicture}
\quad
\begin{tikzpicture}
\draw[step=0.5cm,color=black, line width=1.5] (0,0) grid (2.5,2.5);
\node at (0.25,2.25) {\queen};
\node at (1.25,1.75) {\queenB};
\node at (1.75,0.25) {\queen};
\node at (2.25,1.25) {\queen};
\end{tikzpicture}
\quad
\begin{tikzpicture}
\draw[step=0.5cm,color=black, line width=1.5] (0,0) grid (2.5,2.5);
\node at (0.75,2.25) {\queen};
\node at (2.25,1.75) {\queen};
\node at (1.25,1.25) {\queenB};
\node at (0.25,0.75) {\queen};
\node at (1.75,0.25) {\queen};
\end{tikzpicture}
\quad
\begin{tikzpicture}
\draw[step=0.5cm,color=black, line width=1.5] (0,0) grid (2.5,2.5);
\node at (0.25,2.25) {$3$};
\node at (0.75,2.25) {$2$};
\node at (1.25,2.25) {$2$};
\node at (1.75,2.25) {$2$};
\node at (2.25,2.25) {$3$};

\node at (0.25,1.75) {$2$};
\node at (0.75,1.75) {$2$};
\node at (1.25,1.75) {$3$};
\node at (1.75,1.75) {$2$};
\node at (2.25,1.75) {$2$};

\node at (0.25,1.25) {$2$};
\node at (0.75,1.25) {$3$};
\node at (1.25,1.25) {$4$};
\node at (1.75,1.25) {$3$};
\node at (2.25,1.25) {$2$};

\node at (0.25,0.75) {$2$};
\node at (0.75,0.75) {$2$};
\node at (1.25,0.75) {$3$};
\node at (1.75,0.75) {$2$};
\node at (2.25,0.75) {$2$};

\node at (0.25,0.25) {$3$};
\node at (0.75,0.25) {$2$};
\node at (1.25,0.25) {$2$};
\node at (1.75,0.25) {$2$};
\node at (2.25,0.25) {$3$};
\end{tikzpicture}
\end{center}
    \caption{Queen trapping for Island 5}
    \label{fig:queentrapping5}
\end{figure}

Having completed the answer sheets for Island 5, she decided to meet up with the King to let him know she was finished. ``Then, \textit{finally}, I'll choose my outfit for tomorrow!'' she told herself. ``I wonder if the King is anywhere close to finishing his answer sheets?''

Instead, she found the King relaxing and drinking beer on his recliner, not working at all. ``I'm done with my answer sheets,'' she said. ``Have you finished yours?'' she asked dubiously. 

To her surprise, the King replied, ``We didn't. Instead, we estimated how many queen sisters you might need on larger islands. For example, if the enemy queen attacks $x$ counties, you always need at least $\frac{x}{12}$ queens for Island $n$.''

``Why is that?'' the Queen asked. 

The King replied, ``We can prove that two queens can attack no more than 12 counties together. Each queen attacks in four directions. For every pair of these directions that are not parallel to each other, there is at least one intersection between them. The four directions from the enemy queen are each intersected by at most three directions from the other queen. So there are at most 12 counties two queens can attack together. And since each sister queen can see a maximum of 12 counties that the enemy queen sees, we need at least $\frac{x}{12}$ queens for Island $n$.''

The King had doubts, ``12 counties are a lot. Other chess pieces can attack much fewer counties together. Is this number achievable?'' 

``Yes, we are powerful!'' boasted the Queen and showed an example in Figure~\ref{fig:queenintersecting8} of 12 counties under attack. 

\begin{figure}[!ht]
    \begin{center}
        \begin{tikzpicture}[scale=1]
            \fill [gray] (0,2) rectangle (0.5,2.5);
            \fill [gray] (0.5,0.5) rectangle (1,1);
            \fill [gray] (0.5,2) rectangle (1,2.5);
            \fill [gray] (0.5,3.5) rectangle (1,4);
            \fill [gray] (1,1) rectangle (1.5,1.5);
            \fill [gray] (1,2) rectangle (1.5,2.5);
            \fill [gray] (1.5,1.5) rectangle (2,2);
            \fill [gray] (1.5,2.5) rectangle (2,3);
            \fill [gray] (2,0) rectangle (2.5,0.5);
            \fill [gray] (2,1.5) rectangle (2.5,2);
            \fill [gray] (2,3) rectangle (2.5,3.5);
            \fill [gray] (2.5,1.5) rectangle (3,2);\node at (0.75, 1.75) {\queenB};
            \node at (2.25, 2.25) {\queen};
            \draw[step=0.5cm,color=black, line width=1.5] (0,0) grid (4,4);
        \end{tikzpicture}
    \end{center}
    \caption{Maximum mutual attacks by two queens}
    \label{fig:queenintersecting8}
\end{figure}
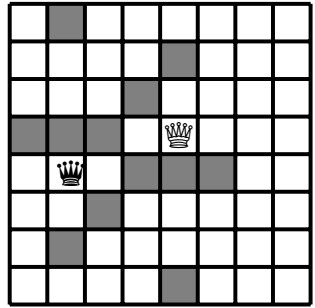

``I see,'' said the King, ``But how big is $x$?

The queen replied, ``If the enemy queen is in the corner, her attacking power is the worst: she attacks $3n-3$ counties. When the enemy queen is in the center, her attacking power is the best: she attacks $4n-4$ counties when $n$ is odd. When $n$ is even, there is no center county, but the maximum attacking power for an enemy queen is in the four center counties. In this case she attacks $4n-5$ counties. So $x$ ranges between $3n-3$ and $4n-4$.''

``On the other hand, you can never need more than $n-1$ sister queens, as you can never even place more than $n$ queens without attacking each other,'' the King adds.

``But this doesn't prove that it is always possible to trap the enemy Queen on larger islands!'' the Queen exclaimed. ``On Islands 1 to 4, it's not even guaranteed that my sisters will be able to trap the enemy Queen!''

``I have a proof! It is stated in the well-known $n$-queens problem that we can arrange $n$ non-attacking queens on chessboards of sizes greater than 5. Each queen in such a configuration is trapped by the other queens. Similar to the non-attacking rook formation, the queens must be in different rows and columns. If a queen moves out of her spot in the configuration, she must be in the same row or column with at least one other queen, so the enemy queen cannot escape anywhere in this configuration.''

``Ha-ha! But do you know whether a solution exists where the enemy queen is located in a particular county?'' contradicted the Queen.

``I am sure it is possible,'' argued the King. For the regular 8-by-8 chessboard, there are 92 different ways to arrange non-attacking queens. Inevitably, one of these solutions has a queen at a county we need.''

A pawn interjected, ``This is true; I checked.''

``But I know that for a 6-by-6 board, there is only one solution up to symmetries, and these solution doesn't have any queens on diagonals,'' the Queen claimed.

``We checked that it is possible to trap the enemy Queen in every county on islands up to 7'' said the King.

``Let's hope that the enemy Queen won't attack other islands,'' the King says gravely. ``But we do know that we'll need not more than 4 of your sisters tomorrow at Island 5.''

``I am tired, and I do not want to think about this anymore,'' the Queen concluded and retired for the day.

\section{A Mathematical Conclusion}

\subsection{Surveying number}

In Section~\ref{sec:surveying} the students study the minimum number of days to survey a given island for a given chess piece. We can call this number \textit{the surveying number}.

We can explain the surveying number in terms of graph theory. Consider a graph $G = G(n,P) = (V, E)$  corresponding to an $n$-by-$n$ island and a chess piece $P$. Each county is a vertex. Two counties are connected if the chess piece $P$ can move from one county to the other.

Then the surveying number $\sigma(G)$ is the smallest number of vertices forming a path such that all other vertices have an edge connecting them to the path. The vertices on such path form a dominating set, where a \textit{dominating set} for a graph $G$ is a subset $D$ of vertices $V$ such that every vertex not in $D$ is adjacent to at least one member of $D$. Therefore, the surveying number can't be smaller than the domination number, where the \textit{domination number} $\gamma(G)$ is the number of vertices in the smallest dominating set for $G$: $\sigma(G) \geq \gamma(G)$. The domination sets on chess graphs were extensively studied. See, for example \cite{Fricke1995, Haynes1998}.

In mathematical terms, the surveying number is the length of the shortest path such that its vertices form a dominating set.

In the chess adventure story, the students started with the surveying number for the Knight. They proved it for small islands of sizes up to 7 and showed interesting surveying examples for larger islands. Table~\ref{tab:knightsurveying number} compares the fastest surveying number the students found with the knights graph's domination number from sequence A006075 in the OEIS \cite{OEIS}. The shoelace examples show that the asymptotic growth of the knight surveying number is bounded by $\frac{2n^2}{7}$.

\begin{table}[ht!]
\begin{center}
    \begin{tabular}{|r|c|c|c|c|c|c|c|c|c|c|c|c|c|c|c|}
\hline
    \textbf{IS}& 1  & 2 &3  &4  &5  &6&7 &8 &9 &10&11&12&13&14&15 \\
    \textbf{FT}& 1  &N/A&N/A&7  &7  &8&11&20&25&33&36&43&47&52&68 \\
    \textbf{DN}& 1  & 4 &4  &4  &5  &8&10&12&14&16&21&24&28&32&36\\
\hline
     \end{tabular}
\end{center}
\caption{Knight surveying and domination numbers}
\label{tab:knightsurveying number}
\end{table}

The rook surveying number is the same as the domination number. It is $n$ for Island $n$. This is not surprising as there exists a rook dominating set that forms a path on the rook graph.

The bishop domination number is $n$, \cite{Fricke1995}. For example, on odd boards, one can put bishops in the middle column. On even boards, on can put bishops on one of the two center columns. The surveying number must be larger as the bishops forming the dominating set cannot reach each other in one move. Our bishop surveying number is $n-2$ for each color.

The king domination number is $\lceil \frac{n}{3} \rceil^2$, see sequence A075561 in the OEIS \cite{OEIS} and \cite{Fricke1995}. Here is how to place the kings. Start by placing a king diagonally adjacent to the bottom-left square. Then, place a king at every location $3k$ units to the right and $3l$ units north of the king already placed such that $k,l\in \mathbb{N}$. If the island size is divisible by 3, then my students' first spiral method is the same as just connecting these placements by a path. Not surprisingly, for such islands, the formula is $\frac{n^2}{3} - 2$. The zig-zag spiral provides a better surveying number. If we look at small islands, the King and the Knight seem to compete with each other, as seen in Table~\ref{tab:summary-sightseeing}. However, asymptotically, the King's surveying number is $\frac{n^2}{4}$, which is better than the Knight's surveying number.

\begin{table}[ht!]
\centering
    \begin{tabular}{|c|c|c|c|c|c|c|c|c|c|c|c|c|c|c|c|}
    \hline
         \textbf{Islands} &1&2&3&4&5&6&7&8&9&10&11&12&13&14&15 \\
         \hline
\textbf{King} &1&1&1&4&7&10&14&18&23&29&34&40&48&55&62 \\
\textbf{Knight}& 1  &N/A&N/A&7  &7  &8&11&20&25&33&36&43&47&52&68 \\
   \hline
    \end{tabular}
\caption{The surveying number}
\label{tab:summary-sightseeing}
\end{table}

The students didn't study the queens much. This is because the queen's domination number was studied a lot. Or maybe, this is because the queen is the most challenging chess piece to study. The queen's domination number is still not known. It is known that the domination number is linear in $n$ and is bounded by \cite{Watkins2012}:
\[\frac {n-1}{2}\leq \gamma(n)\leq n-\left\lfloor \frac {n}{3}\right\rfloor.\]

Interestingly, another special queen domination number is more related to this story: the diagonal domination number, which was studied in \cite{CockayneHedetniemi1986}. The \textit{diagonal domination number} $d(G)$ is the smallest queen's dominating set, where all the queens are placed on the main diagonal. It is known that \cite{CockayneHedetniemi1986} for sufficiently large $n$, the diagonal domination number exceeds the domination number. As a diagonal is a good surveying path, we can say that 
\[\gamma(G) \leq \sigma(G) \leq d(G).\]

We can express a sequence of queens placement on the diagonal as a set of numbers $K$ in the range 1--$n$, where each number represent both coordinates. The theorem in \cite{CockayneHedetniemi1986} states that the set $K$ corresponds to the diagonal dominating set if and only if its complement is a midpoint-free even-sum set. Here \textit{midpoint-free} means that the set doesn't contain an average of any two of its elements. \textit{Even-sum} means that each sum of a pair of elements is even. As a consequence, the elements are either all even or all odd. Let us denote the diagonal domination number as $d_n$. As we just discussed, $d_n$ equals $n$ minus the size of the largest midpoint-free even-sum set in the range 1--$n$.

Sequence A003002 in the OEIS \cite{OEIS} counts the size of the largest subset $S_n$ of the numbers in the range 1--$n$ which does not contain a 3-term arithmetic progression. To calculate $d_n$, we need to look separately at even and odd numbers. If we have a midpoint-free even-sum subset of even numbers in the range 1--$n$, then we can subtract 1 from all the numbers and get a midpoint-free even-sum subset of odd numbers in the range 1--$n$. Thus, we can assume that the largest subset is odd.

Given a midpoint-free even-sum subset of odd numbers in the range 1--$n$, we can add 1 and divide by 2 to get a midpoint-free subset in the range 1--$\lfloor \frac{n+1}{2} \rfloor$. Thus, the sequence $E_n$ can be expressed through $S_n$:

The subset of odd numbers is not less than the subset of even numbers, and its size is $\lfloor \frac{n+1}{2} \rfloor$. If we denote the largest midpoint-free even-sum set in the range 1--$n$ as $E_n$, then
\[E_n = \textrm{A}003002\left( \left\lfloor \frac{n+1}{2} \right\rfloor \right).\]

For the values of $n$ between 1 and 10 inclusive, the result is:
\[1,\ 1,\ 2,\ 2,\ 2,\ 2,\ 3,\ 3,\ 4,\ 4,\ \ldots.\]

We need to keep in mind, that $d_1 = 1$. Otherwise the diagonal domination number for $n > 1$ is
\[d_n = n - \textrm{A}003002\left( \left\lfloor \frac{n+1}{2} \right\rfloor \right).\]
We submitted this sequence to the OEIS database \cite{OEIS} as sequence A358062. The first 10 values are:
\[0,\ 1,\ 1,\ 2,\ 3,\ 4,\ 4,\ 5,\ 5,\ 6,\ \ldots.\]

For Island 9, the queen domination number, as well as the queen diagonal domination number, is 5. Thus, the surveying number is also 5, and the proposed plan is the best. For Island 10, however, the queen domination number is 5, while the diagonal domination number is 6, see sequence A075458 (domination number for queen's graph) in the OEIS and references there.

\subsection{Trapping number}

In Section~\ref{sec:trapping}, we study the minimum number of non-attacking chess pieces that can be placed on the board so that one of the pieces --- the enemy piece --- can't move anywhere without being attacked. We look at the trapping number in terms of graph theory.

A natural first thought is to compare the trapping number to the non-attacking number $\alpha(G)$. Where the non-attacking number is the largest number of given chess pieces that can be placed on the board without mutually attacking each other. Suppose, for example, we have an arrangement of non-attacking queens. As the number of queens is the maximum, we can't add any more queens. That means that all the cells of the chessboard that are not occupied by queens are under attack. Suppose one of the queens in this arrangement is the enemy Queen. That means the enemy Queen is trapped. One might conclude that $\alpha(n) - 1$ is the upper bound for the trapping number. This conclusion has a flaw. We need to be able to trap an enemy Queen on any cell of the board. This means we need to be able to find a non-attacking arrangement with a queen on a given cell. If this is always possible, then the non-attacking number is an upper bound for the trapping numbers.

A natural second thought is to compare the trapping number to the domination number $\gamma(G)$. Suppose we have queens arranged in the dominating set on a chessboard. If we place the enemy Queen in one of the empty cells, the enemy Queen almost satisfies the condition of being trapped. Indeed, any cell where the enemy Queen can move to is under attack. The two caveats are the following: the enemy Queen is itself under attack, and the other queens might attack each other. Still, the intuition suggests that the trapping number might be smaller than the domination number, as, in domination, the queens need to attack all the cells on the board, while in trapping, the queens need to attack only a portion of the cells. So we expect that the trapping numbers will not exceed the domination number. 

By the way, one can see that the domination number can't exceed the non-attacking number: 
\[\gamma(G) \leq \alpha(G).\]

To define the trapping number, we need some more terms. A set of vertices is called \textit{independent} if no two vertices in the set are adjacent. We are given a set of elements $U = \{1,2,...,n\}$ called \textit{the universe}. A collection $S$ of $m$ sets whose union equals the universe is called a \textit{set cover}. \textit{The set cover problem} is to identify the smallest sub-collection of $S$ whose union equals the universe.

Given the island size and the type of chess pieces we use, we choose the corresponding graph $G$. Then we fix a vertex $v$ of the graph, which is equivalent to placing an enemy piece on the board. In our case, the set $U$ consists of the counties that the enemy chess piece can move to: the vertices adjacent to $v$ in graph $G$. The sets in the collection $S$ are defined as follows. Consider a county $C$ not occupied by an enemy piece. A set corresponding to this county $S_c$ is the set of the vertices in $U$ that are covered by $C$. Finding the smallest cover set for the collection $S$ is equivalent to trapping the enemy piece, disregarding the extra condition that the pieces shouldn't attack each other. Given that we study trapping where our chess pieces do not attack each other, we can say that we are solving the cover problem on the condition that the set in the cover has to consist of isolated vertices. In other words, we look for an independent set cover.

To summarize, the \textit{trapping number} is the size of the smallest set $T$ of vertices of $G$ not containing $v$ satisfying the following properties:
\begin{enumerate}
\item The set $T \cap v$ is independent.
\item All the vertices adjacent to $v$ are covered by the sets of neighbors of vertices in $T$.
\end{enumerate}

The students found a way to approach the lower bound for the trapping number. The idea is to look at the maximum number $m$ of vertices adjacent to two vertices in a graph $G$. This number is always finite. It is the largest for the queens' graph, where it can reach 12.

This approach explains why the trapping number behaves differently for pieces that move locally, like knights and kings, and for other pieces that can jump across the chessboard. For local pieces, the trapping number is bounded by a fixed constant. For non-local pieces, the trapping number is linear in the size of the board $n$.

The students completely analyzed local pieces. They showed that we never need more than 4 knights to trap a knight, and we never need more than 4 kings to trap a king.

They also proved the easiest case: the rooks. The trapping number is always $n-1$, which is one less than the domination number. One of the ways to trap a rook is to use Latin squares. Suppose there is an enemy rook on a particular cell of the chessboard of size $n$. Consider a Latin square of size $n$. The enemy cell corresponds to a particular number in a Latin square. All other entries of the same number in a Latin square can be used as trapping placements for white rooks.

The trapping number for bishops is more complicated. The problem is that for any board, if a bishop is placed on the diagonals and the cells adjacent to the diagonals, except the ones on the border, then such a bishop can't be trapped. Otherwise, the bishop's trapping number equals the length of the longest diagonal occupied by the bishop minus 1. When the enemy bishop can be trapped, the trapping number is between $\lfloor \frac{n}{2} \rfloor$ and $n-2$.

The trapping number for queens is even more complicated. The students proved the lower bound of $\lceil \frac{n-1}{4} \rceil$. If it is always possible to trap, then the upper bound has to be $n-1$. The students implicitly conjectured that on boards of sizes 5 and more, the enemy queen can always be trapped. The proof of this conjecture is left as an exercise for the reader.

\section{Acknowledgments}

We are grateful to the PRIMES STEP program and its director, Dr.~Slava Gerovitch, for giving us this opportunity to conduct research.

\end{document}